\documentclass[11pt,a4paper]{amsart} 
\usepackage{amssymb,amscd,amsmath, young,mathrsfs,mathabx}
\usepackage{mathrsfs, mathdots,rotating,multirow} 
\usepackage[mathcal]{eucal} 
\usepackage{float}
\usepackage[usenames]{color}
\usepackage{soul}
\usepackage{longtable}

\theoremstyle{plain}
\newtheorem{thm}{Theorem}[section]
\newtheorem{lem}[thm]{Lemma}
\newtheorem{pro}[thm]{Proposition}
\newtheorem{cor}[thm]{Corollary}
\newtheorem*{claim*}{Claim}

\newtheorem{con}[thm]{Conjecture}

\theoremstyle{remark}
\newtheorem{rem}[thm]{Remark}

\newtheorem{exm}[thm]{Example}

\newtheorem{dfn}[thm]{Definition}

\numberwithin{equation}{section}

\numberwithin{table}{section}

\newcommand{\mcDtwo}{\mathcal{D}^{(2)}}

\newcommand{\N}{\mathbb{N}}
\newcommand{\Z}{\mathbb{Z}}
\newcommand{\Q}{\mathbb{Q}}

\newcommand{\C}{\mathbb{C}}

\newcommand{\mff}{\mathfrak{f}}

\newcommand{\mfp}{\mathfrak{p}}
\newcommand{\mfP}{\mathfrak{P}}

\newcommand{\lri}{\mathfrak{o}}

\newcommand{\ol}{\overline}

\newcommand{\Gri}{\ensuremath{\mathcal{O}}}

\renewcommand{\epsilon}{\varepsilon}

\renewcommand{\phi}{\varphi}
\renewcommand{\theta}{\vartheta}
\newcommand{\mcA}{\mathcal{A}}
\newcommand{\mcH}{\mathcal{H}}
\newcommand{\mcO}{\mathcal{O}}

\newcommand{\rarr}{\rightarrow}

\newcommand{\ideal}{\triangleleft}
\newcommand{\zideal}{\zeta^{\triangleleft}}

\DeclareMathOperator{\Spec}{Spec}
\DeclareMathOperator{\gr}{gr}

\DeclareMathOperator{\rk}{rk}
\DeclareMathOperator{\SL}{SL}
\DeclareMathOperator{\GL}{GL}

\DeclareMathOperator{\real}{Re}

\def \topo {\textup{top}}
\def \redu {\textup{red}}
\def \vn {\varnothing}

\def \bfo {{\bf 1}}
\def \bft {{\bf 2}}
\def \ul {\underline}
\def \zetagrid{\zeta^{\triangleleft_\textup{gr}}_L(s)}

\def \zetaFgridtop{\zeta^{\triangleleft_\textup{gr}}_{\mathfrak{f}_{3,3},\textup{top}}}
\def \zetaFgridred{\zeta^{\triangleleft_\textup{gr}}_{\mathfrak{f}_{3,3},\textup{red}}}

\def  \wo {h} 
\def \la {\langle} 
\def \ra {\rangle} 
\def \grL {\textup{gr}L}
\def \idealgr {\triangleleft_\textup{gr}}

\def \bfo {{\bf 1}}

\def \bfX {{\bf X}}

\def \mcD {\ensuremath{\mathcal{D}}}

\def \Fq {\ensuremath{\mathbb{F}_q}}
\def \Zp  {\mathbb{Z}_p}

\author{Seungjai Lee} \address{ Mathematical Institute, Oxford
  University, OX2 6GG, United Kingdom}\curraddr{National Institute for
  Mathematical Sciences, Daejeon 305-811, South
  Korea}\email{seung.lee@mansfield.ox.ac.uk}

\author{Christopher Voll} \address{Fakult\"at f\"ur Mathematik,
  Universit\"at Bielefeld, D-33501 Bielefeld, Germany}
\email{C.Voll.98@cantab.net}


\makeatother

\begin{document}

\title[Enumerating graded ideals in free nilpotent Lie
rings]{Enumerating graded ideals in graded rings associated to free
  nilpotent Lie rings}

\date{\today}
\begin{abstract}
We compute the zeta functions enumerating graded ideals in the graded
Lie rings associated with the free $d$-generator Lie rings
$\mff_{c,d}$ of nilpotency class $c$ for all $c\leq2$ and for
$(c,d)\in\{(3,3),(3,2),(4,2)\}$. We apply our computations to obtain
information about $\mfp$-adic, reduced, and topological zeta
functions, in particular pertaining to their degrees and some special
values.
\end{abstract}

\maketitle

\thispagestyle{empty}

\section{Introduction}

\subsection{Enumerating graded ideals in graded Lie rings}

Let $R$ be the ring of integers of a number field or the completion of
such a ring at a nonzero prime ideal. Let $L$ be a nilpotent $R$-Lie
algebra of nilpotency class $c$, free of finite rank over $R$, with
lower central series $(\gamma_{i}(L))_{i=1}^{c}$.  For $i \in
\{1,\dots,c\}$, set $L^{(i)}:=\gamma_{i}(L)/\gamma_{i+1}(L)$.  The
\emph{associated graded $R$-Lie algebra} is
$\grL=\bigoplus_{i=1}^{c}L^{(i)}$. An $R$-ideal $I$ of $\grL$ (of
finite index in $\grL$) is \emph{graded} or \emph{homogeneous} if it
is generated by homogeneous elements or, equivalently, if
$I=\bigoplus_{i=1}^{c}(I\cap L^{(i)})$. In this case we write
$I\triangleleft_{\gr}\grL$. We define the \emph{graded ideal zeta
  function of $L$} as the Dirichlet generating series
\begin{equation}\label{def:graded.ideal.z.f.}
\zetagrid=\sum_{I\triangleleft_{\gr}\grL}|\grL:I|^{-s},
\end{equation}
enumerating graded ideals in $\grL$ of finite index in $\grL$. Here
$s$ is a complex variable; our assumptions on $L$ guarantee that
$\zetagrid$ converges on a complex half-plane. (For a similar
definition in a slightly more general setting,
see~\cite[Section~3.1]{Rossmann/16}).

Assume now that $R=\mcO$ is the ring of integers of a number field.
For a (nonzero) prime ideal $\mfp\in\Spec(\mcO)$ we write
$\mcO_{\mfp}$ for the completion of $\mcO$ at $\mfp$, a complete
discrete valuation ring of characteristic zero and residue field
$\mcO/\mfp$ of cardinality $q$, 
say. 
Primary decomposition yields the Euler product
\begin{equation*}
  \zeta^{\idealgr}_{L}(s)
  =\prod_{\mfp\in\Spec(\mcO)}\zeta^{\triangleleft_\textup{gr}}_{L(\Gri_{\mfp})}(s),\label{equ:euler}
\end{equation*}
expressing the ``global'' zeta function $\zetagrid$ as an infinite
product of ``local'', or $\mfp$-adic ones. By slight abuse of notation
we denote here by $\Spec(\Gri)$ the set of nonzero prime ideals
of~$\Gri$. Each individual Euler factor is a rational function in the
parameter $q^{-s}$. In fact, \eqref{equ:euler} is an Euler product of
\emph{cone integrals} in the sense of~\cite{duSG/00}
(cf.\ \cite[Theorem~3.3]{Rossmann/16}), and the far-reaching results
of this paper -- regarding both the factors and the product's analytic
properties -- apply. The fact that the analysis of \cite{duSG/00} is
restricted to the case $R=\Z$ is insubstantial for this conclusion.

\subsection{Main results}

In the present paper we are concerned with graded ideal zeta functions
of free nilpotent Lie rings of finite rank. Given $c\in\N$ and
$d\in\N_{\geq2}$, let $\mathfrak{f}_{c,d}$ be the free nilpotent Lie
ring of nilpotency class $c$ on $d$ Lie generators. One may identify
$\mathfrak{f}_{c,d}$ with the quotient of the free $\Z$-Lie algebra
$\mathfrak{f}_{d}$ on $d$ generators by the $c+1$-th term
$\gamma_{c+1}(\mathfrak{f}_{d})$ of its lower central series. For
$i\in \{1,\dots,c\}$, the $\Z$-rank of the $i$-th
lower-central-series quotient
$\gamma_{i}(\mathfrak{f}_{c,d})/\gamma_{i+1}(\mathfrak{f}_{c,d})$ is
given by the Witt function
\begin{equation}
W_{d}(i):=\frac{1}{i}\sum_{j|i}\mu(j)d^{i/j},\label{def:witt}
\end{equation}
where $\mu$ denotes the M\"obius function;
cf.\ \cite[Satz~3]{Witt/37}. Hence
$\textup{rk}_{\Z}\left(\mathfrak{f}_{c,d}\right)=\sum_{i=1}^{c}W_d(i)$. Given
a commutative ring $R$ -- in this paper always of the form $\Gri$ or
$\Gri_{\mfp}$ as above -- we write $\mff_{c,d}(R) =
\mff_{c,d}\otimes_{\Z}R$, considered as an $R$-Lie algebra.

For $c=1$, the (graded ideal) zeta function of the free abelian Lie
ring $\mathfrak{f}_{1,d}(\Gri)$, enumerating all finite index
$\Gri$-sublattices of $\Gri^{d}$, is well known to be equal to
\begin{equation}
\zeta_{\mathfrak{f}_{1,d}(\mcO)}^{\idealgr}(s)=\prod_{i=1}^{d}\zeta_{K}(s-i+1),\label{equ:abelian}
\end{equation}
where $\zeta_{K}$ denotes the Dedekind zeta function of $K$;
cf.\ \cite[Proposition~1.1]{GSS/88}. The Euler
product~\eqref{equ:euler} reflects the well-known Euler product
$\zeta_{K}(s) = \prod_{\mfp\in\Spec(\Gri)} (1-|\Gri/\mfp|^{-s})^{-1}$:
\begin{equation}\label{equ:abel.euler} 
 \zeta_{\mathfrak{f}_{1,d}(\mcO)}^{\idealgr}(s)=\prod_{\mfp\in\Spec(\Gri)}
 \prod_{i=1}^d\frac{1}{1-|\Gri/\mfp|^{i-1-s}} =
 \prod_{\mfp\in\Spec(\Gri)} \zeta_{\Gri_{\mfp}^d}(s),\textup{ say}.
\end{equation}

For $c=2$, the second author computed in \cite{Voll/05a} the
\emph{ideal} zeta functions enumerating ($\Gri$-)ideals of finite
index in the rings $\mathfrak{f}_{2,d}(\Gri)$.  (The paper only
discusses the case $R=\Z$, but its computations carry over --
\emph{mutatis mutandis} -- to the case of general number rings.) We
compute the \emph{graded ideal} zeta functions
$\zeta_{\mathfrak{f}_{2,d}(\Gri)}^{\idealgr}(s)$ for all $d\geq2$ in
Section~\ref{sec:c=2} and those for $(c,d)\in\{(3,2), (4,2)\}$ in
Section~\ref{sec:d=2} of the current paper.

The paper's most involved result is the computation, in
Section~\ref{sec:c=d=3}, of the graded ideal zeta function of
$\mathfrak{f}_{3,3}(\mcO)$. To this end we compute an explicit formula
for $\zeta_{\mathfrak{f}_{3,3}(\lri)}^{\idealgr}(s)$, valid for all
finite extension $\lri$ of the $p$-adic integers $\Zp$, where $p$ is a
prime, viz.\ a local ring of the form $\lri = \Gri_{\mfp}$ for a
nonzero prime ideal $\mfp$ of $\Gri$ lying above~$p$.

\begin{thm}\label{thm:main} There exists an explicitly determined
rational function $W^{\idealgr}_{3,3}\in\Q(X,Y)$ such that, for all
primes $p$ and all finite extensions $\lri$ of~$\Zp$, with residue
cardinality $q$, 
\[
\zeta_{\mathfrak{f}_{3,3}(\lri)}^{\idealgr}(s)=W^{\idealgr}_{3,3}(q,q^{-s}).
\]
It may be written as $W^{\idealgr}_{3,3} = N_{3,3}/D_{3,3}$ where
\begin{align*}
D_{3,3}(X,Y) =& (1 - Y)(1 - XY)(1 - X^2 Y)(1 - X^3 Y^4 )(1 - XY^5 )(1 -
X^2 Y^5 )(1 - Y^6 )\\&(1 - X^7 Y^7 )(1 - X^8 Y^7 ) (1 - X^8 Y^8 )(1 -
X^{12} Y^8 )(1 - X^6 Y^9 )(1 - X^{15} Y^9 )\\&(1 - X^{16} Y^{10} )(1 -
X^{15} Y^{11} )(1 - X^{12} Y^{12})(1 - X^7 Y^{13} )(1 - Y^{14}),
\end{align*}
a polynomial of degree $115$ in $X$ and $131$ in $Y$, and
$N_{3,3}\in\Q[X,Y]$ is a polynomial of degree $81$ in $X$ and $108$
in~$Y$.

The rational function $W^{\idealgr}_{3,3}$ satisfies the functional equation
\begin{equation}
W^{\idealgr}_{3,3}(X^{-1},Y^{-1})=X^{34}Y^{23}W^{\idealgr}_{3,3}(X,Y).\label{equ:funeq}
\end{equation}
\end{thm}

We note that the Witt function $W_{3}$ (cf.~\eqref{def:witt}) takes
the values $(W_{3}(1),W_{3}(2),W_{3}(3))=(3,3,8)$ and that
$115-81=34=\binom{3}{2}+\binom{3}{2}+\binom{8}{2}$ and
$131-108=23=3\cdot3+2\cdot3+1\cdot8$; cf.\ Conjecture~\ref{con:funeq}.

Our proof of Theorem~\ref{thm:main} yields $W^{\idealgr}_{3,3}$ as a
sum of 15 explicitly given summands, listed essentially in
Section~\ref{subsec:2dim}. We do not reproduce the ``final'' outcome
of this summation here, as the numerator $N_{3,3}$ of
$W^{\idealgr}_{3,3}$ fills several pages. We do record, however,
several corollaries of the explicit formula for~$W^{\idealgr}_{3,3}$.

The first corollary concerns analytic properties of the global zeta
function~$\zeta_{\mathfrak{f}_{3,3}(\mcO)}^{\idealgr}(s)$.

\begin{cor}
  The global graded ideal zeta function
  $\zeta_{\mathfrak{f}_{3,3}(\mcO)}^{\idealgr}(s)$ converges on
  $\{s\in\C \mid \real(s) > 3\}$ and may be continued meromorphically
  to $\{s \in \C \mid \real(s) > 14/9\}$.
\end{cor}

This follows from two observations. Firstly, the product
$\prod_{\mfp\in\Spec(\Gri)} D_{3,3}(q,q^{-s})^{-1}$ -- a product of
finitely many translates of the Dedekind zeta function $\zeta_K(s)$ --
has abscissa of convergence $3$ and may be continued meromorphically
to the whole complex plain. Secondly, the product
$\prod_{\mfp\in\Spec(\Gri)} N_{3,3}(q,q^{-s})$ may be continued
meromorphically to $\{s \in \C \mid \real(s) > 14/9\}$. Indeed, if
$N_{3,3}(X,Y)=1 + \sum_{i\in I}\alpha_i X^{a_i}Y^{b_i}$ for a finite
index set $I$ and $\alpha_i\in\Z\setminus\{0\}$, $a_i\in\N_0$,
$b_i\in\N$, then $\prod_{\mfp\in\Spec(\Gri)} N_{3,3}(q,q^{-s})$ may be
continued meromorphically to $\{ s\in \C \mid \real(s) > \beta\}$,
where $\beta := \max\{\frac{a_i}{b_i} \mid i\in I \}$;
cf.\ \cite[Lemma~5.5]{duSWoodward/08}. That $\beta = 14/9$ follows
from inspection of~$N_{3,3}$.

The second corollary concerns the \emph{reduced graded ideal zeta
  function} $$\zetaFgridred(Y):=W^{\idealgr}_{3,3}(1,Y)\in\Q(Y);$$
cf.\ Section~\ref{subsec:reduced}. This concept was introduced
in~\cite{Evseev/09}, albeit not in the context of graded ideal zeta
functions. We expect that our ad hoc definition will fit into a general
definition of reduced graded ideal zeta functions along the lines
of~\cite{Evseev/09}.

\begin{cor} The reduced graded ideal zeta function $\zetaFgridred(Y)$
  satisfies
\begin{equation}
\zetaFgridred(Y)=\frac{N_{3,3,\redu}(Y)}{(1-Y)^{3}(1-Y^{3})(1-Y^{4})^{2}(1-Y^{5})\prod_{i=8}^{14}(1-Y^{i})},\label{equ:red}
\end{equation}
where $N_{3,3,\redu}(Y)\in\Z[Y]$ is equal to
\begin{align*}
 & 1+2Y^{3}+4Y^{4}+5Y^{5}+16Y^{6}+34Y^{7}+53Y^{8}+77Y^{9}+98Y^{10}+121Y^{11}+182Y^{12}+\\
 & 302Y^{13}+483Y^{14}+712Y^{15}+953Y^{16}+1187Y^{17}+1425Y^{18}+1689Y^{19}+2046Y^{20}+\\
 & 2579Y^{21}+3298Y^{22}+4162Y^{23}+5059Y^{24}+5826Y^{25}+6398Y^{26}+6894Y^{27}+\\
 & 7475Y^{28}+8270Y^{29}+9265Y^{30}+10260Y^{31}+11041Y^{32}+11529Y^{33}+11745Y^{34}+\\
 & 11798Y^{35}+11811Y^{36}+11811Y^{37}+11798Y^{38}+11745Y^{39}+11529Y^{40}+11041Y^{41}+\\
 & 10260Y^{42}+9265Y^{43}+8270Y^{44}+7475Y^{45}+6894Y^{46}+6398Y^{47}+5826Y^{48}+\\
 & 5059Y^{49}+4162Y^{50}+3298Y^{51}+2579Y^{52}+2046Y^{53}+1689Y^{54}+1425Y^{55}+\\
 & 1187Y^{56}+953Y^{57}+712Y^{58}+483Y^{59}+302Y^{60}+182Y^{61}+121Y^{62}+98Y^{63}+\\
 & 77Y^{64}+53Y^{65}+34Y^{66}+16Y^{67}+5Y^{68}+4Y^{69}+2Y^{70}+Y^{73}.
\end{align*}
\end{cor} It seems remarkable that $W^{\idealgr}_{3,3}(1,Y)$ has a
simple pole at $Y=1$ of order $14 = 3+3+8$, the $\Z$-rank of
$\mathfrak{f}_{3,3}$, and that $N_{3,3,\redu}$ has {nonnegative},
unimodal coefficients. This is consistent with the speculation that
$W^{\idealgr}_{3,3}(1,Y)$ is in fact the Hilbert-Poincar\'e series of
a $14$-dimensional graded algebra, associated to $\mathfrak{f}_{3,3}$
in a natural way; cf.\ Remark~\ref{rem:red}. Note that the right hand
side of \eqref{equ:red} is not in lowest terms.  The palindromic
symmetry of the coefficients of $N_{3,3,\redu}$ is implied by the
functional equation~\eqref{equ:funeq}.

The third corollary concerns another limiting object of the local
graded ideal zeta functions
$\zeta^{\triangleleft_\textup{gr}}_{\mff_{3,3}(\lri)}(s)$, viz.\ the
\emph{topological graded ideal zeta function}
$\zeta^{\triangleleft_\textup{gr}}_{\mff_{3,3},\topo}(s)$;
cf.\ Section~\ref{subsec:topo}. In \cite{Rossmann/15}, Rossmann
introduced and studied so-called topological zeta functions associated
to a range of Dirichlet generating series. Informally speaking,
topological zeta functions may be viewed as suitably defined limits of
local zeta functions. Whilst topological versions of zeta functions
such as $\zeta^{\triangleleft_\textup{gr}}_{\mff_{3,3}(\lri)}(s)$ have
not yet been studied specifically, one may define
$\zeta^{\triangleleft_\textup{gr}}_{\mff_{3,3},\topo}(s)$ as the
coefficient of $(q-1)^{-14}$ in the expansion of
$\zeta^{\triangleleft_\textup{gr}}_{\mff_{3,3}(\lri)}(s)$ in~$q-1$;
cf.~\cite[Definition~5.13]{Rossmann/15}.

\begin{cor} The topological graded ideal zeta function
  $\zetaFgridtop(s)$ satisfies
\[
\zetaFgridtop(s)=\frac{(33250s^{4}-81537s^{3}+66573s^{2}-20800s+1920)/56}{D_{3,3,\topo}(s),}
\]
where $D_{3,3,\topo}(s)\in\Z[s]$ is defined to be 
\[
s^{3}(s-1)^{4}(s-2)(3s-2)(4s-3)(5s-1)(5s-2)(7s-8)(2s-3)(3s-5)(5s-8)(11s-15)(13s-7)
\]
\end{cor}

A number of further features of the $\mfp$-adic, topological, and
reduced graded ideal zeta functions associated to $\mff_{3,3}$
considered above seem remarkable. Some of the numerical values in the
following corollary are presented so as to illustrate the general
conjectures in Section~\ref{sec:con} which we extracted from the
paper's explicit computations.

By the \emph{degree} of a rational function $f
= P/Q\in\Q(Y)$ in $Y$ we mean $\deg_Y f = \deg_Y P - \deg_Y Q$.

\begin{cor}\label{cor:f33}
\begin{enumerate}
\item 
\[
\deg_{s}\left(\zetaFgridtop(s)\right)= -14 = - (3+3+8) =
-\rk_{\Z}(\mathfrak{f}_{3,3}),
\]

\item 
\[
s^{-14}\zetaFgridtop(s^{-1})|_{s=0}=\frac{19}{288,288}=(1-Y)^{14}\zetaFgridred(Y)|_{Y=1} \in \Q_{>0},
\]

\item 
\[
s^{3}\zetaFgridtop(s)|_{s=0}=\frac{-1}{70,560}=\frac{(-1)^{2+2+7}3\cdot3\cdot8}{3\cdot6\cdot14\cdot2!\cdot2!\cdot7!},
\]

\item 
\[
\left.\frac{\zeta^{\triangleleft_\textup{gr}}_{\mff_{3,3}(\lri)}(s)}{\zeta_{\lri^{3}}(s)\zeta_{\lri^{3}}(s)\zeta_{\lri^{8}}(s)}\right|_{s=0}=\frac{2}{7}=\frac{3\cdot3\cdot8}{3\cdot6\cdot14}.
\]
\end{enumerate}
\end{cor} 

If the above-mentioned connection between the reduced graded ideal
zeta function and Hilbert-Poincar\'e series of graded algebras were to
hold, the rational number in part (2) -- for which we currently do not
have an interpretation -- were to be interpreted in terms of the
multiplicity of the associated graded algebra; cf., for instance,
\cite[Section~4]{BrunsHerzog/93}.


\subsection{Background, motivation, and methodology}

Recall that $L$ is a nilpotent Lie algebra over the ring of integers
$\Gri$ of a number field. The \emph{ideal zeta function} of $L$ is the
Dirichlet generating series enumerating the $\Gri$-ideals of $L$ of
finite index, viz.\
$$\zeta^{\ideal}_L(s) = \sum_{I \triangleleft L} |L:I|^{-s} =
\prod_{\mfp \in \Spec(\Gri)}\zeta^{\ideal}_{L(\Gri_{\mfp})}(s),$$
where $s$ is a complex variable and, for a (nonzero) prime ideal
$\mfp$ of $\Gri$,
$$\zeta^{\ideal}_{L(\Gri_{\mfp})}(s) = \sum_{I \ideal L(\Gri_{\mfp})}
|L(\Gri_{\mfp}):I|^{-s}$$ enumerates $\Gri_{\mfp}$-ideals of
$L(\Gri_{\mfp})$. Ideal zeta functions of nilpotent Lie rings are
well-studied relatives of the graded ideal zeta functions studied in
the present paper. One of the main results of \cite{GSS/88}, which
introduced the former zeta functions, establishes the rationality of
each of the Euler factors $\zeta^{\ideal}_{L(\Gri_{\mfp})}(s)$ in
$q^{-s}$, where $q=|\Gri:\mfp|$; cf.\ \cite[Theorem~3.5]{GSS/88}. For
numerous examples of ideal zeta functions of nilpotent Lie rings, see
\cite{duSWoodward/08}.

By the Mal'cev correspondence, ideal zeta functions of nilpotent Lie
rings are closely related to the \emph{normal subgroup zeta functions}
enumerating finite index normal subgroups of finitely generated
nilpotent groups. In particular, given $c,d\in\N$, almost all
(i.e.\ all but finitely many) of the Euler factors of
$\zeta^{\ideal}_{\mff_{c,d}}(s)$ coincide with those of the normal
subgroup zeta function of the free class-$c$-nilpotent $d$-generator
group $F_{c,d}$; cf.\ \cite[Section~4]{GSS/88}. The study of the
normal subgroup growth of free nilpotent groups is connected with the
enumeration of finite $p$-groups up to isomorphism;
cf.\ \cite{duS/02}.

Informally speaking, graded ideal zeta functions of nilpotent Lie
rings may be seen as ``approximations'' of their ideal zeta
functions. Indeed, almost all Euler factors
$\zeta^{\idealgr}_{L(\Gri_{\mfp})}(s)$ actually enumerate a sublattice
of the lattice of ideals enumerated by
$\zeta^{\triangleleft}_{L(\Gri_{\mfp})}(s)$. In general, this
approximation is quite coarse. In nilpotency class $c\leq 2$, however,
the problems of computing ideal zeta functions and \emph{graded} ideal
zeta functions are closely related, as the following example shows.

\begin{exm}\label{exm:ideal.z.f}
  Assume that $L$ is nilpotent of class~$2$, with isolated commutator
  ideal $L'$ such that $\rk_{\mcO}(L/L')=d$.  Then, essentially by
  \cite[Lemma~6.1]{GSS/88},
\begin{equation}
\zeta_{L}^{\triangleleft}(s)=\sum_{\Lambda_{1}\leq L/L'}|L/L':\Lambda_{1}|^{-s}\sum_{[\Lambda_{1},L]\leq\Lambda_{2}\leq L'}|L':\Lambda_{2}|^{d-s},\label{equ:gss.lemma.6.1}
\end{equation}
whereas
\begin{equation}
\zetagrid=\sum_{\Lambda_{1}\leq L/L'}|L/L':\Lambda_{1}|^{-s}\sum_{[\Lambda_{1},L]\leq\Lambda_{2}\leq L'}|L':\Lambda_{2}|^{-s}.\label{equ:gen.rewrite}
\end{equation}
In the special cases $L=\mff_{2,d}$, the proximity between
\eqref{equ:gss.lemma.6.1} and \eqref{equ:gen.rewrite} explains the
proximity between the explicit formulae recorded in
Theorems~\ref{thm:f2d.ideal} and \ref{thm:f2d.idealgr} of the current
paper.
\end{exm}

In higher nilpotency classes, we are not aware of any such simple
parallels. In the realm of free nilpotent Lie rings of class greater
than two, explicit computations of ideal zeta functions seem all but
unfeasible. In particular, we do not know of formulae for the ideal
zeta functions of the Lie rings $\mff_{3,3}(\Gri)$ and
$\mff_{4,2}(\Gri)$.

Numerous questions regarding ideal zeta functions have analogues
regarding the ``approximating'' \emph{graded} ideal zeta functions,
and one may speculate that the latter are easier to answer than the
former. On p.~188 of \cite{GSS/88}, Grunewald, Segal, and Smith
formulate, for example, a conjecture which would imply that, for any
$c,d\in\N$ there exists a rational function $W^{\ideal}_{c,d}(X,Y) \in
\Q(X,Y)$ such that, for almost all $p$ and all finite extensions
$\lri$ of $\Zp$, $\zeta^{\ideal}_{\mff_{c,d}(\lri)}(s) =
W^{\ideal}_{c,d}(q,q^{-s})$. This consequence is known to hold for
$c\leq 2$ and $(c,d) = (3,2)$ but wide open in general, including the
cases $(c,d)\in\{(3,3),(4,2)\}$. Our Conjecture \ref{con:uniformity}
-- which is verified by the explicit computations in the current paper
in particular for $c = 2$ and $(c,d)\in\{(3,3),(3,2),(4,2)\}$ -- may
be viewed as a ``graded'' version of this conjecture.

In \cite{Rossmann/15}, Rossmann formulates a number of conjectures on
certain special values of $\mfp$-adic and topological zeta functions,
pertaining in particular to ideal zeta functions of nilpotent Lie
rings. Our Conjectures~\ref{con:top.inf}, \ref{con:top.zero} and
\ref{con:p-ad.zero} are ``graded'' counterparts.

In \cite[Theorem~4.4]{Voll/16}, the second author proved a local
functional equation for the generic Euler factors of the ideal zeta
functions $\zeta^{\ideal}_{\mff_{c,d}(\Gri)}(s)$ upon inversion of the
prime for all $c\in\N$ and $d\in\N_{\geq 2}$. Our results suggest that
this phenomenon also appears for graded ideal zeta functions of free
nilpotent Lie rings; cf.\ Conjecture~\ref{con:funeq}.

\medskip
All our computations owe their feasibility to the fact that, for the
parameter values considered, viz.\ $c=2$ and
$(c,d)\in\{(3,3),(3,2),(4,2)\}$, the enumeration of graded ideals is
equivalent to the enumeration of various flags of lattices in free
$\lri$-modules which depend only, and in a linear fashion, on the
lattices' \emph{elementary divisor types}. Our computations rely on a
simple polynomial formula, due to Birkhoff, for the numbers of
$\lri$-submodules of given type in a finite $\lri$-module of given
type; cf.\ Proposition~\ref{pro:birkhoff}. To obtain closed formulae
for the relevant graded ideal zeta functions we need to organize the
enumeration of infinitely many values of Birkhoff's formula in a
manageable way. We meet this challenge by organizing pairs of
partitions, encoding two lattices' elementary divisor types, by their
\emph{overlap type}, formally one of finitely many multiset
permutations, viz.\ words in two letters (each with multiplicity);
cf.\ Section~\ref{subsec:overlap}. For $(c,d)=(3,3)$, for instance, we
are led to consider 15 specific words of length $11$ in the
alphabet~$\{\bfo, \bft\}$;
cf.\ Table~\ref{tab:dyck.words}. Restricting the relevant enumerations
to a fixed word yields formulae which are products of $q$-multinomial
coefficients and so-called \emph{Igusa functions};
cf.\ Section~\ref{subsec:igusa}. In Theorem~\ref{thm:2D-reduction} we
compute such formulae for each of the 15 words in question. Similar
methodology has been applied before to compute ideal zeta functions of
Heisenberg Lie rings over number rings (\cite{SV1/15}) and free
nilpotent Lie rings of class $2$ (\cite{Voll/05a}).

That we are able to reduce our computations to combinatorial
considerations with partitions is essentially owed to the fact that,
for the parameters $(c,d)$ considered, the automorphism groups of the
free Lie rings $\mff_{c,d}$ act transitively on the lattices of given
elementary divisor types in relevant sections of $\mff_{c,d}$. This
allows us to assume that the relevant lattices are generated by
multiples of elements of \emph{Hall bases};
cf.\ Section~\ref{subsec:hall}. That a lattice may always be generated
by multiples of \emph{some} linear basis follows, of course, from the
elementary divisor theorem; that (subsets of) Hall bases may be used
is a lucky consequence of the existence of ``many'' automorphisms, as
we now explain. The automorphism group of $\mff_{c,d}(\lri)$ contains
a copy of $\GL_d(\lri)$ which allows us to perform arbitrary
invertible $\lri$-linear transformations of Lie generators
$x_1,\dots,x_d$ of $\mff_{c,d}(\lri)$. For $c=2$, this observation --
which was already exploited in the proof of \cite[Theorem~2]{GSS/88}
-- in conjunction with Birkhoff's formula is sufficient to compute the
(graded) ideal zeta function of $\mff_{c,d}(\lri)$;
cf.\ Section~\ref{sec:c=2}. On the Lie commutators $[x_i,x_j]$, $1
\leq i < j \leq n$, these linear transformations act via the exterior
square representation. In general, the image of the natural map
$\bigwedge^{2}\GL_{d}(\lri) \rarr \GL_{\binom{d}{2}}(\lri)$ is rather
small. The map is surjective, however, if $d\leq 3$, and an
isomorphism if $d=3$. This facilitates our computations for
$(c,d)\in\{(3,3),(3,2)\}$. If $d=2$, then the above-mentioned copy of
$\GL_2(\lri)$ even induces the full automorphism group of the
$\lri$-module generated by the weight-$3$-commutators
$[[x_1,x_2],x_1]$ and $[[x_1,x_2],x_2]$, which we exploit for $(c,d) =
(4,2)$.

\subsection{Notation}

Given $n\in\N=\{1,2,\dots\}$, we write $[n]$ for $\{1,2,\dots,n\}$. We
write $\N_{\geq 2}$ for $\{2,3,\dots,\}\subseteq \N$. Given a subset
$I\subseteq\N$, we write $I_{0}$ for $I\cup\{0\}$. If
$I\subseteq[n-1]$, we denote $n-I=\{n-i\mid i\in I\}$. Given
$a,b\in\N_{0}$, we denote the interval $\{a+1,\dots,b\}$ by $]a,b]$
    and the interval $\{a+1,\dots,b-1\}$ by~$]a,b[$. We write $2^{S}$
        for the power set of a set $S$.  The notation
        $I=\{i_{1},\dots,i_{h}\}_{<}$ for a subset of $I\subset\N_{0}$
        indicates that $i_{1}<i_{2}<\dots<i_{h}$. Similarly,
        $(\lambda_1,\dots,\lambda_n)_{\geq}\in\N_0^n$ denotes a
        partition with non-ascending parts $\lambda_1 \geq \dots \geq
        \lambda_n \geq 0$.

We write $\lri$ for a compact discrete valuation ring of
characteristic zero, i.e.\ a finite extension of the ring $\Zp$ of
$p$-adic integers or, equivalently, a ring of the form $\Gri_{\mfp}$,
the completion of $\Gri$, the ring of integers of a number field $K$,
at a nonzero prime ideal $\mfp$ of $\Gri$. We write $q$ for the
cardinality of the residue field of $\lri$ and $p$ for its residue
characteristic. We set $t=q^{-s}$, where $s$ is a complex
variable. $\zeta_{K}$ is the Dedekind zeta function of~$K$.

\section{General preliminaries}
Let $c\in\N$, $d\in\N_{\geq 2}$, and set $\mff=\mff_{c,d}$ be as
above. The following is analogous to \cite[Lemma~6.1]{GSS/88}.
\begin{lem}\label{lem:sum} Let $p$ be a prime
  and $\lri$ be a finite extension of $\Zp$. For each $\Lambda_2\leq
  \mff(\lri)^{(2)}$ let
\begin{equation}\label{def:X}
 (X(\Lambda_2)\oplus \mff(\lri)^{(2)})/\Lambda_2 = Z((\mff(\lri)^{(1)}
  \oplus \mff(\lri)^{(2)})/\Lambda_2).
\end{equation}
Then
\begin{align}
  \zeta^{\triangleleft_\textup{gr}}_{\mff(\lri)}(s) &
  =\sum_{\substack{\Lambda_{1}\leq
      \mff(\lri)^{(1)},\dots,\;\Lambda_{c}\leq
      \mff(\lri)^{(c)}\\ \forall
      i\in\,]1,c]:\;[\Lambda_{i-1},\mff(\lri)^{(1)}]\leq\Lambda_{i} }
  }\prod_{i=1}^{c}|\mff(\lri)^{(i)}:\Lambda_{i}|^{-s}\nonumber\\ &=\zeta_{\lri^{d}}(s)\sum_{\substack{\Lambda_{2}\leq
        \mff(\lri)^{(2)},\dots,\;\Lambda_{c}\leq
        \mff(\lri)^{(c)}\\ \forall
        i\in\,]2,c]:\;[\Lambda_{i-1},\mff(\lri)^{(1)}]\leq\Lambda_{i}
    }
      }|\mff(\lri)^{(1)}:X(\Lambda_{2})|^{-s}\prod_{i=2}^{c}|\mff(\lri)^{(i)}:\Lambda_{i}|^{-s}.\label{equ:X}
  \end{align}  
\end{lem}

\begin{proof}
  A graded additive sublattice $\Lambda \leq \gr \mff(\lri)$
  determines and is determined by the sequence
  $(\Lambda_1,\dots,\Lambda_c)$, where $\Lambda_i := \Lambda \cap
  \mff(\lri)^{(i)}$ for all $i\in[c]$. We have $\Lambda \idealgr
  \mff(\lri)$ if and only if $[\Lambda_{i-1}, \mff(\lri)^{(1)}] \leq
  \Lambda_i$ for all $i\in\,]1,c]$. This proves the first
      equality. The second follows from the definition of
      $X(\Lambda_2)$ given in \eqref{def:X}, noting that
      $[\Lambda_1,\mff(\lri)^{(1)}] \leq \Lambda_2$ if and only if
      $\Lambda_1 \leq X(\Lambda_2)$ and that, for each $\Lambda_2\leq \mff(\lri)^{(2)}$,
$$\sum_{\Lambda_1 \leq X(\Lambda_2)} | \mff(\lri)^{(1)}:\Lambda_1|^{-s} = \zeta_{\lri^d}(s) | \mff(\lri)^{(1)}:X(\Lambda_2)|^{-s}.\qedhere$$
\end{proof}
\subsection{Birkhoff's formula}

Given a pair of partitions $(\sigma,\tau)$, let $\alpha(\tau,\sigma;q)$
denote the number of torsion $\lri$-modules of type $\tau$ contained
in a fixed torsion $\lri$-module of type $\sigma$. Clearly $\alpha(\sigma,\tau;q)=0$
unless $\tau\leq\sigma$. Notice that, for $\tau=(\tau_{1},\dots,\tau_{n})_{\geq}$,
\[
\alpha(\tau_{1}^{(n)},\tau;q)=\#\{\Lambda\leq\lri^{n}\mid\lri^{n}/\Lambda\cong\oplus_{j=1}^{n}\lri/\mfp^{\tau_{j}}\}.
\]
The following explicit general formula for $\alpha(\sigma,\tau;q)$
is attributed to Birkhoff in~\cite{Butler/87}.

\begin{pro}[Birkhoff] \label{pro:birkhoff} Let $\tau\leq\sigma$ be
  partitions, with dual partitions $\tau^{\prime}\leq\sigma^{\prime}$.
  Then
\[
\alpha(\sigma,\tau;q)=\prod_{k\geq1}q^{\tau_{k}^{\prime}(\sigma_{k}^{\prime}-\tau_{k}^{\prime})}\binom{\sigma_{k}^{\prime}-\tau_{k+1}^{\prime}}{\sigma_{k}^{\prime}-\tau_{k}^{\prime}}_{q^{-1}}\in\Z[q].
\]
\end{pro}

\subsection{Igusa functions and their functional equations}\label{subsec:igusa}

\begin{dfn}(\cite[Definition~2.5]{SV1/15})\label{def:igusa} Let
  $\wo\in\N$. Given variables $\bfX=(X_{1},\dots,X_{\wo})$ and~$Y$, we
  set
\begin{align*} 
 I_{\wo}(Y;\bfX) &
 =\frac{1}{1-X_{h}}\sum_{I\subseteq[\wo-1]}\binom{\wo}{I}_{Y}\prod_{i\in
   I}\frac{X_{i}}{1-X_{i}}\;\in\Q(Y,X_{1},\dots,X_{\wo}),\\ I_{\wo}^{\circ}(Y;\bfX)
 &
 =\frac{X_{h}}{1-X_{h}}\sum_{I\subseteq[\wo-1]}\binom{\wo}{I}_{Y}\prod_{i\in
   I}\frac{X_{i}}{1-X_{i}}\;\in\Q(Y,X_{1},\dots,X_{\wo}).
\end{align*}
We set $I_{0}(Y)=I_{0}^{\circ}(Y)=1\in\Q(Y)$. \end{dfn} Note that
$I_{1}(Y;X)=\frac{1}{1-X}$ and $I_{1}^{\circ}(Y;X)=\frac{X}{1-X}$. We
will make repeated use of the functional equations
(\cite[Proposition~4.2]{SV1/15})
\begin{align}
  I_{h}(Y^{-1};\bfX^{-1}) &
  =(-1)^{h}X_{h}Y^{-\binom{h}{2}}I_{h}(Y;\bfX),\nonumber\\ I_{h}^{\circ}(Y^{-1};\bfX^{-1})
  &
  =(-1)^{h}X_{h}^{-1}Y^{-\binom{h}{2}}I_{h}^{\circ}(Y;\bfX).\label{equ:funeq.igusa}
\end{align}

\begin{rem}\label{rem:igusa}
The ``symmetry centres'' $(-1)^{h}X_{h}Y^{-\binom{h}{2}}$
resp.\ $(-1)^{h}X_{h}^{-1}Y^{-\binom{h}{2}}$ in the functional
equations \eqref{equ:funeq.igusa} depend on $Y$ and $X_{h}$, but not
on $X_{1},\dots,X_{h-1}$.
\end{rem}
Throughout this paper, we always substitute $q^{-1}$ for $Y$ and hence
write $I_{\wo}(\bfX)$ instead of $I_{\wo}(q^{-1};\bfX)$ and
$I_{\wo}^{\circ}(\bfX)$ instead of $I_{\wo}^{\circ}(q^{-1};\bfX)$.

Igusa functions are ubiquitous in the theory of zeta functions of groups and rings. Note, for instance, that
$$\zeta_{\lri^d}(s) = I_d((q^{(d-j)j}t^j)_{j\in[d]}) =
\frac{1}{\prod_{j=0}^{d-1}(1-q^jt)};$$ cf.\ \eqref{equ:abel.euler}.

\subsection{Hall bases for free nilpotent Lie rings}\label{subsec:hall}
Let $x_1,\dots,x_d$ be Lie generators for $\mff = \mff_{c,d}$. A
\emph{Hall basis for $\mathfrak{f}_{c,d}$} (on $x_{1},\dots,x_{d}$) is
a $\Z$-basis $\mcH_{c,d}$ for $\mathfrak{f}_{c,d}$ which may be
constructed by selecting inductively certain \emph{basic commutators}
in the $x_{i}$; see \cite{Hall/50} for details. Note that
$|\mcH_{c,d}| = \sum_{i=1}^c W_d(i)$, where $W_d$ denotes the Witt
function~\eqref{def:witt}.

In Table~\ref{tab:hall} we record some Hall bases which are relevant
for the current paper. Here we adopt the standard abbreviation
$x_{i_1}x_{i_2}x_{i_3} \dots x_{i_r}$ for left-normed commutators
$[\dots [[x_{i_1},x_{i_2}],x_{i_3}],\dots,x_{i_r}]$. In the case that
$d=2$ (resp.\ $d=3$), we write $x$, $y$ (and $z$) instead of $x_1$,
$x_2$ (and $x_3)$.

\begin{table}[htb!]
\centering
\begin{tabular}{c|l|l}
  $(c,d)$ & Hall basis $\mcH_{c,d}$& $(W_d(i))_{i\in[c]}$\\\hline
  $(1,d)$ & $\{x_1,\dots,x_d\}$ & $(d)$\\
  $(2,d)$ & $\{x_{1},\dots,x_{d},\; x_{i}x_{j}\mid1\leq i<j\leq d\}$& $(d, \binom{d}{2})$\\
  $(3,3)$ & $\{x,y,z,\; xy, xz,yz, \; xyx, xzx,yzx, xyy, xzy, yzy, xzz, yzz\}$& $(3,3,8)$\\
  $(3,2)$ & $\{ x,y, \; xy, \;xyx, xyy \}$ & $(2,1,2)$\\
  $(4,2)$ & $\{ x,y,\; xy,\; xyx,xyy, \; xyxx,xyyx,xyyy \}$ & $(2,1,2,3)$
\end{tabular}
\medskip
\caption{Some Hall bases}
\label{tab:hall}
\end{table}

\subsection{Reduced zeta functions}\label{subsec:reduced}

In \cite{Evseev/09}, Evseev introduced a certain ``limit as $\mfp\rarr
1$'' of various local zeta functions. Informally speaking, the idea is
to exploit the fact that the coefficients of these generating
functions enumerate $\Fq$-rational points of constructible sets. The
desired limit is obtained by replacing these constructible sets by
their Euler-Poincar\'e characteristics. The resulting rational
functions are called \emph{reduced zeta functions};
cf.~\cite[Section~3]{Evseev/09}.  This paper does not define reduced
\emph{graded} ideal zeta functions explicitly, but provides a
theoretical framework that allows this quite readily. If $c,d$ are
such that
$\zeta^{\idealgr}_{\mff_{c,d}(\lri)}(s)=W_{c,d}^{\idealgr}(q,q^{-s})$
for a rational function $W_{c,d}^{\idealgr}(X,Y)\in\Q(X,Y)$, for
almost all primes $p$ and all finite extensions $\lri$ of $\Zp$, then
$\zeta^{\triangleleft_\textup{gr}}_{\mff_{c,d},\redu}(s) =
W_{c,d}^{\idealgr}(1,Y) \in \Q(Y)$.  Evseev proved in
\cite[Section~4]{Evseev/09} that reduced ideal zeta functions of Lie
rings with so-called \emph{nice} and \emph{simple} bases are
Hilbert-Poincar\'e series enumerating integral points in rational
polyhedral cones. The Hall bases $\mcH_{2,d}$ given in
Table~\ref{tab:hall}, for instance, have this property. Comparing
\eqref{equ:gss.lemma.6.1} and \eqref{equ:gen.rewrite} one sees that,
for nilpotent Lie rings $L$ of nilpotency class $2$,
$\zeta_{L,\redu}^{\idealgr}(Y)=\zeta_{L,\redu}^{\ideal}(Y)$; cf.\ also
Remark~\ref{rem:f32}.

\subsection{Topological zeta functions}\label{subsec:topo}

Another means of defining limits of $\mfp$-adic zeta functions are
topological zeta functions. If $c,d$ are such that
$\zeta^{\idealgr}_{\mff_{c,d}(\lri)}(s) =
W^{\idealgr}_{c,d}(q,q^{-s})$ for a rational function
$W^{\idealgr}_{c,d}(X,Y)\in\Q(X,Y)$, for almost all primes $p$ and all
finite extensions $\lri$ of~$\Zp$, then
$\zeta^{\idealgr}_{\mff_{c,d},\topo}(s)$ is simply the coefficient of
$(q-1)^{-r}$ in the expansion of $W^{\idealgr}_{c,d}(q,q^{-s})$ in
$q-1$, where~$r = \rk_{\Z}(\mff_{c,d})$.

\begin{exm} For $a\in\N_{0}$ and $b\in\N$,
$
\frac{1}{1-q^{a-bs}}=\frac{1}{bs-a}(q-1)^{-1}+O(1).
$
Thus 
\[
\zeta_{\mathfrak{f}_{1,d}(\lri)}^{\idealgr}(s) =
\zeta_{\lri^{d}}(s)=\frac{1}{s(s-1)\dots(s-d+1)}(q-1)^{-d}+O((q-1)^{-d+1}),
\]
whence 
\[
\zeta_{\mathfrak{f}_{1,d},\topo}^{\idealgr}(s)=\frac{1}{s(s-1)\dots(s-d+1)}.
\]
\end{exm}

More generally, \cite[Definition~5.13]{Rossmann/15} applies to any
system of local zeta functions of Denef type over a number field, in
the sense of \cite[Definition~5.7]{Rossmann/15}. Examples of such
systems are families of $\mfp$-adic zeta functions arising from
families of the form $(W(q,q^{-s}))_{\mfp\in\Spec(\Gri)}$, for
suitable $W(X,Y)\in\Q(X,Y)$ (and, as usual, $q=|\Gri/\mfp|$ for
$\mfp\in\Spec(\Gri)$); cf.\ also \cite{Rossmann/15a}. All the local
graded ideal zeta functions considered in the current paper fit into
such ``uniform'' families. We expect this phenomenon to be universal
in the context of free nilpotent Lie rings;
cf.\ Conjecture~\ref{con:uniformity}.

For a formal and far more general definition of topological zeta
functions we refer to \cite[Section~5]{Rossmann/15}. In
\cite[Section~8]{Rossmann/15}, Rossmann collects a number of
intriguing conjectures about analytic properties of topological zeta
functions associated to various group- and ring-theoretic counting
problems. We expect most of these conjectures to have analogues in the
realm of topological graded ideal zeta functions. Motivated by the
computation of various topological graded ideal zeta functions of free
nilpotent Lie rings made throughout the current paper, we make a
number of such conjectures in Section~\ref{sec:con}.

\section{Proof of Theorem~\ref{thm:main} ($(c,d)=(3,3)$)}

\label{sec:c=d=3}

Let $\mff = \mathfrak{f}_{3,3}$ be the free nilpotent Lie ring on $3$
generators of nilpotency class~$3$.  Let $p$ be a prime and $\lri$ be
a finite extension of~$\Zp$, with uniformizer~$\pi$.  Note that
$\mff(\lri)^{(1)} \cong \lri^3 \cong \mff(\lri)^{(2)}$ and
$\mff(\lri)^{(3)} \cong \lri^8$. In order to parameterize lattices in
$\mff(\lri)^{(2)}$ and $\mff(\lri)^{(3)}$ we denote by
\begin{equation*}
\mu=(\mu_{1},\mu_{2},\mu_{3})_{\geq}\quad \textup{ and }\quad\nu=(\nu_{1},\dots,\nu_{8})_{\geq}
\end{equation*}
integer partitions of at most $3$ and $8$ parts
respectively, and set
\begin{equation*}
\ol{\mu}=(\mu_{1}^{(2)},\mu_{2}^{(3)},\mu_{3}^{(3)}) = (\mu_1,\mu_1,\; \mu_2,\mu_2,\mu_2,\; \mu_3,\mu_3,\mu_3)_{\geq}.
\end{equation*}

\begin{pro}\label{pro:setup.2d}
\begin{equation}
  \zeta^{\triangleleft_\textup{gr}}_{\mff_{3,3}(\lri)}(s)=\\
  \zeta_{\lri^3}(s) \sum_{\mu}\alpha(\mu_1^{(3)},\mu;q)q^{-s(3\mu_1 +
    2 \mu_2 +
    \mu_3)}\sum_{\nu\leq\ol{\mu}}\alpha(\ol{\mu},\nu;q)q^{-s\sum_{k=1}^{8}\nu_{k}}.\label{equ:setup.2d}
\end{equation}
\end{pro}

\begin{proof}
 Our starting point is \eqref{equ:X} in Lemma~\ref{lem:sum}, which in
 this case reads
$$ \zeta^{\idealgr}_{\mff_{3,3}(\lri)}(s) = \zeta_{\lri^3}(s)
 \sum_{\Lambda_2 \leq \mff(\lri)^{(2)}} |
 \mff(\lri)^{(1)}:X(\Lambda_2)|^{-s} |\mff(\lri)^{(2)}:\Lambda_2|^{-s}
 \sum_{\substack{\Lambda_3\leq
     \mff(\lri)^{(3)},\\ [\Lambda_2,\mff(\lri)^{(1)}]\leq
     \Lambda_3}}|\mff(\lri)^{(3)}:\Lambda_3|^{-s}.$$

Let $\mu=(\mu_{1},\mu_{2},\mu_{3})_{\geq}$ be a partition.  There are
$\alpha(\mu_1^{(3)},\mu;q)$ lattices $\Lambda_{2}\leq
\mff(\lri)^{(2)}$ whose elementary divisor type with respect to
$\mff(\lri)^{(2)}$ is given by $\mu$. Clearly
$|\mff(\lri)^{(2)}:\Lambda_{2}|=q^{\sum_{j=1}^{3}\mu_{j}}$. We claim
that $|\mff(\lri)^{(1)}:X(\Lambda_2)| = q^{2\mu_1 + \mu_2}$. Indeed,
lattices such as $\Lambda_2$ may be parametrized by their elementary
divisor type and a coset $\alpha \Gamma_\mu\in
\SL_3(\lri)/\Gamma_{\mu}$ of a certain stabilizer subgroup
$\Gamma_{\mu} \leq \SL_3(\lri)$. Then
$|\mff(\lri)^{(1)}:X(\Lambda_2)|$ is the index of the lattice of
solutions to the simultaneous congruences
$$(x_1,x_2,x_3) \left( \begin{matrix} 0& \alpha_{1j} &
  \alpha_{2j}\\-\alpha_{1j} & 0 & \alpha_{3j}\\ -\alpha_{2j} &
  -\alpha_{3j} &0\end{matrix} \right) \equiv 0 \bmod \mfp^{\mu_j}$$
  for $j=1,2,3$; cf.\ \cite[Section~2.2]{Voll/04} for details. This
  index clearly is $q^{2\mu_1 + \mu_2}$.

After a change of generators for $\mff(\lri)$ if necessary, we may
assume that
\begin{equation*}
\Lambda_{2}=\mfp^{\mu_{1}} xy \oplus\mfp^{\mu_{2}} xz
\oplus\mfp^{\mu_{3}} yz.\label{equ:type2}
\end{equation*}
At this point we crucially use the fact that $d=3$, as
$\GL_{d}(\lri)\cong\bigwedge^{2}\GL_{d}(\lri)$ if and only if $d=3$.
We obtain
\begin{multline}
  [\Lambda_{2},\mff(\lri)^{(1)}]=
  \\\la\pi^{\mu_{1}}xyx,\pi^{\mu_{1}}xyy,\ul{\pi^{\mu_{1}}xyz},\pi^{\mu_{2}}xzx,\ul{\pi^{\mu_{2}}xzy},\pi^{\mu_{2}}xzz,
  \ul{\pi^{\mu_{3}}yzx},\pi^{\mu_{3}}yzy,\pi^{\mu_{3}}yzz\ra_{\lri}.\label{equ:lambda2}
\end{multline}
The Jacobi identity involving the three underlined terms is the only
nontrivial relation between these terms. Indeed, the relation
$xyz+zxy+yzx=0$ implies that
\begin{equation*}
  \la\pi^{\mu_{1}}xyz,\pi^{\mu_{2}}xzy,\pi^{\mu_{3}}yzx\ra_{\lri} = \la\pi^{\mu_{1}}(xzy -
  yzx),\pi^{\mu_{2}}xzy,\pi^{\mu_{3}}yzx\ra_{\lri}
  =\mfp^{\mu_{2}}xzy\oplus\mfp^{\mu_{3}}yzx.
\end{equation*}
Hence \eqref{equ:lambda2} implies that 
\begin{multline*}
[\Lambda_{2},\mff(\lri)^{(1)}]= \\ \mfp^{\mu_{1}}xyx
\oplus\mfp^{\mu_{1}}xyy\oplus
\mfp^{\mu_{2}}xzx\oplus\mfp^{\mu_{2}}xzy\oplus\mfp^{\mu_{2}}xzz\oplus
\mfp^{\mu_{3}}yzx\oplus\mfp^{\mu_{3}}yzy\oplus\mfp^{\mu_{3}}yzz,
\end{multline*}
whence $[\Lambda_{2},\mff(\lri)^{(1)}]$ has type $\ol{\mu}$ with
respect to $\mff(\lri)^{(3)}$. 
This shows that the number of lattices $\Lambda_{3}\leq \mff(\lri)^{(3)}$
whose elementary divisor type with respect to $\mff(\lri)^{(3)}$ is given
by a partition $\nu=(\nu_{1},\dots,\nu_{8})_{\geq}$ and which satisfy
$[\Lambda_{2},\mff(\lri)^{(1)}]\leq\Lambda_{3}$ is equal to
$\alpha(\ol{\mu},\nu;q)$.  Each such lattice satisfies
$|\mff(\lri)^{(3)}:\Lambda_{3}|=q^{-s\sum_{k=1}^{8}\nu_{k}}$.
\end{proof}

\subsection{Overlap types and $2$-dimensional
  words}\label{subsec:overlap}

Our approach to computing the right hand side of \eqref{equ:setup.2d}
is similar to the one taken in \cite{SV1/15} to compute local factors
of the ($\Z$-) ideal zeta functions of Lie rings of the form
$\mff_{2,2}(\Gri)$. 

To compute the right hand side of \eqref{equ:setup.2d} we carry out a
case distinction with respect to the finitely many ways in which the
partitions $\ol{\mu}$ and $\nu$ may ``overlap''. To be precise, let
$\mu$ and $\nu$ be partitions of $3$ and $8$ parts, respectively,
satisfying $\nu\leq\ol{\mu}$. There are uniquely determined numbers
$r\in\N_{0}$ and $M_{i},N_{i}\in\N$ ($i\in[r-1]$), such that
\begin{multline*}
\mu_{1}\geq\cdots\geq\mu_{M_{1}}\geq\nu_{1}\geq\cdots\geq\nu_{N_{1}}>\\
\mu_{M_{1}+1}\geq\cdots\geq\mu_{M_{2}}\geq\nu_{N_{1}+1}\geq\cdots\geq\nu_{N_{2}}>\dots\\
\mu_{M_{r-1}+1}\geq\cdots\geq\mu_{3}\geq\nu_{N_{r-1}+1}\geq\cdots\geq\nu_{8}.\label{mult:limi}
\end{multline*}
Define $M_{r}=3$, $N_{r}=8$, and $M_{0}=N_{0}=0$.  We call the
integer sequence $(M_{i},N_{i})_{i\in[r-1]}$ arising the \emph{overlap
  type} of the pair $(\mu,\nu)$. Set
\begin{alignat*}{2}
\widehat{\phantom{x}}:\N_{0} & \rarr\N_{0}, & \quad n & \mapsto
n\cdot(2,3,3):=2\delta_{n\geq1}+3\delta_{n\geq2}+3\delta_{n\geq3},
\end{alignat*}
An overlap type determines and is
determined by the \emph{2D-word}
\[
v=\bfo^{M_{1}}\bft^{N_{1}}\bfo^{M_{2}-M_{1}}\bft^{N_{2}-N_{1}}\dots\bfo^{3-M_{r-1}}\bft^{8-N_{r-1}},
\]
i.e.\ a word on the alphabet $\{\bfo,\bft\}$ of length eleven such
that $\bfo$ occurs three times and $\bft$ eight times, and
$\widehat{M_i}\geq N_{i}$ for $i\in[r-1]$. Observe that the latter
condition is equivalent to $\nu\leq\ol{\mu}$.  In this case we write
$v(\mu,\nu)=v$. We denote by $\mathcal{D}^{(2)}$ the set of 15
2D-words arising in this way.

\begin{rem}
Our notion of 2D-word is related to the classical one of
(2-dimensional) \emph{Dyck word of length $2n$}. The latter are words
in $\{\bfo,\bft\}$, featuring $n$ occurrences of each letter. They may
be used to model the overlap types of two partitions, each of at most
$n$ parts; cf.\ \cite[Section~2.4]{SV1/15}. 2D-words, in contrast,
model the overlaps of two partitions $\nu$ and $\ol{\mu}$; whilst
$\nu$ has genuinely at most $8$ parts, $\ol{\mu}$ has at most $3$
distinct parts, with ties in blocks of respective sizes 2, 3, and~3.

We chose to phrase our results and the supporting notation as similar
to the relevant material in \cite{SV1/15} as possible, leaving the
reader free to concentrate on the crucial technical differences. Our
Lemmata \ref{lem:product B'} and \ref{lem:product A'}, for instance,
are similar to, but subtly different from Lemmata~2.16 and 2.17 in
\cite{SV1/15}. Theorem~\ref{thm:2D-reduction} is analogous to
\cite[Theorem~3.1]{SV1/15}.
\end{rem}

For $v\in\mcDtwo$ we set
\begin{equation}
  D_{v}(q,t):=\sum_{\substack{(\mu,\nu)\\ v(\mu,\nu)=v }
  }\alpha(\mu_{1}^{(3)},\mu;q)\alpha(\ol{\mu},\nu;q)q^{-s(3\mu_1 +
    2\mu_2 + \mu_3 +\sum_{k=1}^{8}\nu_{k})}.\label{def:Dv}
\end{equation}
Proposition~\ref{pro:setup.2d} allows us to write
\begin{equation}
\zeta_{\mff_{3,3}(\lri)}^{\triangleleft_{\textup{gr}}}(s)=\zeta_{\lri^3}(s) \sum_{v\in\mcDtwo}D_{v}(q,t).\label{equ:2dim}
\end{equation}

This section's main result is Theorem~\ref{thm:2D-reduction}, giving a
general formula for the functions~$D_v$. Table~\ref{tab:dyck.words}
lists the 15 words in $\mcD^{(2)}$ together with their overlap types
and an indication where in Section~\ref{subsec:2dim} an explicit
formula for $D_{v}$ may be found.

\begin{table}
  \label{tab:dyck.words} \centering \protect\protect\protect\protect\caption{Overlap types and 2D-words}
\begin{tabular}{|c|l||l|c|}
  \hline 
  $r$  & $(M_i,N_i)_{i\in[r-1]}$  & $v\in\mcD^{(2)}$  & Section \tabularnewline
  \hline 
  \hline 
  1  & --  & $\bfo^{3}\bft^{8}$  &   \ref{subsubsec:2dim.1}\tabularnewline 
   \hline 
  2  & $(2,1)$ & $\bfo^{2}\bft\bfo\bft^{7}$  &   \multirow{5}{*}{\ref{subsubsec:2dim.2}} \tabularnewline 
  2  & $(2,2)$  & $\bfo^{2}\bft^{2}\bfo\bft^{6}$  &    \tabularnewline
  2  & $(2,3)$  & $\bfo^{2}\bft^{3}\bfo\bft^{5}$ &   \tabularnewline
  2  & $(2,4)$  & $\bfo^{2}\bft^{4}\bfo\bft^{4}$  &    \tabularnewline
  2  & $(2,5)$  & $\bfo^{2}\bft^{5}\bfo\bft^{3}$  &    \tabularnewline
  \hline 
  2  & $(1,1)$  & $\bfo\bft\bfo^{2}\bft^{7}$  &    \multirow{5}{*}{\ref{subsubsec:2dim.3}} \tabularnewline
  3  & $(1,1;2,2)$  & $\bfo\bft\bfo\bft\bfo\bft^{6}$  &    \tabularnewline
  3  & $(1,1;2,3)$  & $\bfo\bft\bfo\bft^{2}\bfo\bft^{5}$  &    \tabularnewline
  3  & $(1,1;2,4)$  & $\bfo\bft\bfo\bft^{3}\bfo\bft^{4}$  &    \tabularnewline
  3  & $(1,1;2,5)$  &$\bfo\bft\bfo\bft^{4}\bfo\bft^{3}$  &    \tabularnewline
  \hline 
  2  & $(1,2)$ & $\bfo\bft^{2}\bfo^{2}\bft^{6}$  &    \multirow{4}{*}{\ref{subsubsec:2dim.4}}\tabularnewline
  3  & $(1,2;2,3)$ & $\bfo\bft^{2}\bfo\bft\bfo\bft^{5}$  &    \tabularnewline
  3  & $(1,2;2,4)$ & $\bfo\bft^{2}\bfo\bft^{2}\bfo\bft^{4}$  &    \tabularnewline
  3  & $(1,2;2,5)$ & $\bfo\bft^{2}\bfo\bft^{3}\bfo\bft^{3}$  &    \tabularnewline
  \hline 
 
\end{tabular}
\end{table}

In order to obtain such formulae, it is useful for us to have notation
for the successive differences of the parts of $\mu$ and $\nu$. For
$j\in[8]$ we set
\begin{equation}\label{equ:rj}
r_{j}=\begin{cases}
\nu_{j}-\nu_{j+1} & \text{if }j\notin\left\{ N_{1},\ldots,N_{r}\right\} ,\\
\nu_{N_{i}}-\mu_{M_{i}+1} & \text{if }j=N_{i},
\end{cases}
\end{equation}
where we define $\mu_{4}=0$. Similarly, for $j\in[3]$ we set
\begin{equation}\label{equ:sj}
s_{j}=\begin{cases}
\mu_{j}-\mu_{j+1} & \text{if }j\notin\left\{ M_{1},\ldots,M_{r}\right\} ,\\
\mu_{M_{i}}-\nu_{N_{i-1}+1} & \text{if }j=M_{i}.
\end{cases}
\end{equation}
By definition, $r_{j},\,s_{j}\geq0$ for all $j.$ Note also that
$r_{j}>0$ if $j\in\left\{ N_{1},\,\ldots,\,N_{r-1}\right\} $ and
observe that $\mu_{M_{i}}>\mu_{M_{i}+1}$ and
$\nu_{N_{i}}>\nu_{N_{i}+1}$ for each $i\in\left[r-1\right]$. Set
\begin{alignat*}{2}
\widecheck{\phantom{x}}:\N_{0} & \rarr\N_{0}, & \quad n & \mapsto n\cdot(3,2,1):=3\delta_{n\geq1}+2\delta_{n\geq2}+1\delta_{n\geq3}.
\end{alignat*}
Then 
\[
\begin{aligned}
3\mu_{1}+2\mu_{2}+\mu_{3} &
=\sum_{j=1}^{3}\widecheck{j}s_{j}+\sum_{i=1}^{r}\widecheck{M_{i}}\left(r_{N_{i-1}+1}+\cdots+r_{N_{i}}\right),\\ \nu_{1}+\cdots+\nu_{8}
&
=\sum_{j=1}^{8}jr_{j}+\sum_{i=1}^{r-1}N_{i}\left(s_{M_{i}+1}+\cdots+s_{M_{i+1}}\right).
\end{aligned}
\]
Finally, for each $i\in\left[r\right]$ we define 
\begin{align}
J_{i}^{\mu} & =\left\{ j\in\left[M_{i}-M_{i-1}-1\right]\mid\mu_{M_{i}-j}>\mu_{M_{i}-j+1}\right\} ,\label{eq:0}\\
J_{i}^{\nu} & =\left\{ j\in\left[N_{i}-N_{i-1}-1\right]\mid\nu_{N_{i}-j}>\nu_{N_{i}-j+1}\right\}.\label{eq:-1}
\end{align}

We start with an immediate consequence of
Proposition~\ref{pro:birkhoff}.
\begin{lem}\label{lem:birkhoff2D}
  Let $\mu$ and $\nu$ be partitions of $3$ and $8$ parts,
  respectively, satisfying $\nu\leq\ol{\mu}$. Then
  \begin{equation*}
    \sum_{\mu}\alpha(\mu_1^{(3)},\mu;q)\alpha(\ol{\mu},\nu;q) =\prod_{k\geq1}q^{\mu'_{k}(3-\mu'_{k})+\nu'_{k}(\widehat{\mu'_{k}}-\nu'_{k})}\binom{3-\mu'_{k+1}}{3-\mu'_{k}}_{q^{-1}}\binom{\widehat{\mu'_{k}}-\nu'_{k+1}}{\widehat{\mu'_{k}}-\nu'_{k}}_{q^{-1}}.
  \end{equation*}
\end{lem}

Let $\mu$ and $\nu$ be partitions as in
Lemma~\ref{lem:birkhoff2D}. Let
$\left(M_{i},\,N_{i}\right)_{i\in[r-1]}$ be their overlap type and
$\mu'$ resp.\ $\nu'$ be their dual partitions. For $k\in\N$, set
$$d_k := \mu'_{k}(3-\mu'_{k})+\nu'_{k}(\widehat{\mu'_{k}}-\nu'_{k}).$$
Also, set
\begin{alignat*}{2}
a_j &:= M_{i}\left(3-M_{i}\right)+j\left(\widehat{M_{i}}-j\right),&  \quad j &\in\left]N_{i-1},N_{i}\right],\\
b_j &:= j\left(3-j\right)+N_{i-1}\left(\widehat{j}-N_{i-1}\right),& \quad j & \in\left]M_{i-1},M_{i}\right].
\end{alignat*} 
Note that $j$ determines a unique $i$.

\begin{lem}\label{lem:product B'} 
For $i\in[r]$,
\begin{multline*}
 \prod_{k=\mu_{M_{i}}}^{\mu_{M_{i-1}+1}}q^{d_k}\binom{3-\mu'_{k+1}}{3-\mu'_{k}}_{q^{-1}}\binom{\widehat{\mu'_{k}}-\nu'_{k+1}}{\widehat{\mu'_{k}}-\nu'_{k}}_{q^{-1}}
 \\ =\prod_{j=1}^{M_{i}-M_{i-1}}q^{b_{M_{i-1}+j}s_{M_{i-1}+j}}\binom{M_{i}-M_{i-1}}{J_{i}^{\mu}}_{q^{-1}}\binom{3-M_{i-1}}{3-M_{i}}_{q^{-1}}.
 \end{multline*}
\end{lem}

\begin{proof}
 Since all the indices $k$ appearing in the product on the left hand
 side satisfy $\nu_{N_{i-1}}>\mu_{M_{i-1}+1}\geq
 k\geq\mu_{M_{i}}\geq\nu_{N_{i-1}+1}$, we have
 $\nu'_{k}=\nu'_{k+1}=N_{i-1}$ for all $k$ in the interval
 $[\mu_{M_{i}},\,\mu_{M_{i-1}+1}]$. Thus
 \[ \prod_{k=\mu_{M_{i}}}^{\mu_{M_{i-1}+1}}\binom{\widehat{\mu'_{k}}-\nu'_{k+1}}{\widehat{\mu'_{k}}-\nu'_{k}}_{q^{-1}}=1.
 \]
 For $j\in[M_{i}-M_{i-1}]$ we have $\mu'_{k}=M_{i-1}+j$ when
 $\mu_{M_{i-1}+j+1}<k\leq\mu_{M_{i-1}+j};$ observe that it may be the
 case for some $j$ that no index $k$ satisfies this condition.  As a
 result, we see that for each $j\in[M_{i}-M_{i-1}]$, there are exactly
 $s_{M_{i-1}+j}$ elements $k$ of the segment
 $[\mu_{M_{i}},\,\mu_{M_{i-1}+1}]$ for which
 $\mu'_{k}=M_{i-1}+j$. Now, note that
 \[ \binom{3-\mu'_{k+1}}{3-\mu'_{k}}_{q^{-1}}\neq1 \]
 if and only if $\mu'_{k+1}\neq\mu'_{k}$. This is the case exactly
 when there exists an $i$ such that $\mu_{i}=k$. It follows that if
 $J_{i}^{\mu}=\left\{ j_{i,1},\ldots,j_{i,\gamma_{i}}\right\}$, with
 $j_{i,1}<\ldots<j_{i,\,\gamma_{i}}$, $j_{i,\,0}:=0$,
 $j_{i,\,\gamma_{i}+1}:=M_{i}-M_{i-1}$, then
 
\begin{equation}
\prod_{k=\mu_{M_{i}}}^{\mu_{M_{i-1}+1}}\binom{3-\mu'_{k+1}}{3-\mu'_{k}}_{q^{-1}}=\prod_{m=0}^{\gamma_{i}}\binom{3-M_{i}+j_{i,\,m+1}}{3-M_{i}+j_{i,m}}_{q^{-1}}.\label{eq:mu2}
\end{equation}

We make use of the well-known identity
\[
\binom{\alpha}{\beta}_{X}=\frac{1-X^{\alpha}}{1-X^{a-\beta}}\binom{\alpha-1}{\beta}_{X}
\]
for Gaussian binomial coefficients. Applying it inductively, we see
that for all $m\in[\gamma_{i}-1]$,
\[
\binom{3-M_{i}+j_{i,\,m+1}}{3-M_{i}+j_{i,m}}_{q^{-1}}=\binom{j_{i,m+1}}{j_{i,m}}_{q^{-1}}\frac{\binom{3-M_{i}+j_{i,\,m+1}}{3-M_{i}}_{q^{-1}}}{\binom{3-M_{i}+j_{i,\,m}}{3-M_{i}}_{q^{-1}}}.
\]
Hence the right-hand side of \eqref{eq:mu2} is equal to
\[
	\binom{M_{i}-M_{i-1}}{J_{i}^{\lambda}}_{q^{-1}}\binom{3-M_{i-1}}{3-M_{i}}_{q^{-1}}.
	\]
	Thus 
	\begin{multline*}
	\prod_{k=\mu_{M_{i}}}^{\mu_{M_{i-1}+1}}q^{d_k}\binom{3-\mu'_{k+1}}{3-\mu'_{k}}_{q^{-1}}\binom{\widehat{\mu'_{k}}-\nu'_{k+1}}{\widehat{\mu'_{k}}-\nu'_{k}}_{q^{-1}}
	\\
	=\prod_{j=1}^{M_{i}-M_{i-1}}q^{b_{M_{i-1}+j}s_{M_{i-1}+j}}\binom{M_{i}-M_{i-1}}{J_{i}^{\mu}}_{q^{-1}}\binom{3-M_{i-1}}{3-M_{i}}_{q^{-1}}
	\end{multline*}
	as required. \end{proof}

\begin{lem} \label{lem:product A'} 
 For $i\in[r]$,
 \begin{multline*}
 \prod_{k=\mu_{M_{i}+1}+1}^{\mu_{M_{i}}-1}q^{d_k}\binom{3-\mu'_{k+1}}{3-\mu'_{k}}_{q^{-1}}\binom{\widehat{\mu'_{k}}-\nu'_{k+1}}{\widehat{\mu'_{k}}-\nu'_{k}}_{q^{-1}}\\ =\prod_{j=1}^{N_{i}-N_{i-1}}q^{a_{N_{i-1}+j}r_{N_{i-1}+j}}\binom{N_{i}-N_{i-1}}{J_{i}^{\nu}}_{q^{-1}}\binom{\widehat{M_{i}}-N_{i-1}}{\widehat{M_{i}}-N_{i}}_{q^{-1}}.
 \end{multline*}
\end{lem}

\begin{proof} 
 Note that the product on the left hand side may be empty; this
 happens in the case
 $\mu_{M_{i}}=\nu_{N_{i-1}+1}=\cdots=\nu_{N_{i}}=\mu_{M_{i}+1}+1.$
 Since $\mu_{M_{i}}>k>\mu_{M_{i}+1}$, we have
 $\mu'_{k}=\mu'_{k+1}=M_{i}$ for all $k$ in the interval
 $]\mu_{M_{i}+1},\,\mu_{M_{i}-1}]$. Finally, observe that for
     $j\in[N_{i}-N_{i-1}]$ we have $\nu'_{k}=N_{i-1}+j$ when
     $\nu_{N_{i-1}+j+1}<k\leq\nu_{N_{i-1}+j}$. The claim follows as in
     the proof of the previous lemma. \end{proof}

With these preparations in place we now give a formula for the
functions~$D_{v}$.

\begin{thm}\label{thm:2D-reduction} 
Let
$v=\prod_{i=1}^{r}\left(\bfo^{M_{i}-M_{i-1}}\bft^{N_{i}-N_{i-1}}\right)\in\mathcal{D}^{(2)}$.
Then
\begin{align*}
 D_{v}(q,t)= &
 \prod_{i=1}^{r}\binom{3-M_{i-1}}{3-M_{i}}_{q^{-1}}\binom{\widehat{M_{i}}-N_{i-1}}{\widehat{M_{i}}-N_{i}}_{q^{-1}}\prod_{i=1}^{r}I_{M_{i}-M_{i-1}}\left(Y_{M_{i-1}+1},\ldots,Y_{M_{i}}\right)\\ &
 \cdot\left(\prod_{i=1}^{r-1}I_{N_{i}-N_{i-1}}^{\circ}\left(X_{N_{i-1}+1},\ldots,X_{N_{i}}\right)\right) \cdot I_{8-N_{r-1}}\left(X_{N_{r-1}+1},\ldots,X_{8}\right),
\end{align*}
with numerical data
\begin{alignat}{2}
 X_{j} & =q^{a_{j}}t^{\widecheck{M_{i}}+j}& \quad j
 &\in\left]N_{i-1},N_{i}\right],\label{eq:num.data.X}\\ Y_{j} &
     =q^{b_{j}}t^{\widecheck{j}+N_{i-1}}&\quad j &
     \in\left]M_{i-1},M_{i}\right].\label{eq:num.data.Y}
\end{alignat}
\end{thm}

\begin{proof}	
 Given $\boldsymbol{v}=(v_{1},v_{2},v_{3})\in\mathbb{N}_{0}^{3}$ and
 $\boldsymbol{v}'=(v'_{1},\ldots,v'_{8})\in\mathbb{N}_{0}^{8}$, we
 set, for each $i\in\left[r\right]$,
\[
\begin{aligned}
\text{supp}_{i}^{M}(\boldsymbol{v}) & =\left\{
j\in\left[M_{i}-M_{i-1}-1\right]\mid v_{M_{i-1}+j}>0\right\}
,\\ \text{supp}_{i}^{N}(\boldsymbol{v'}) & =\left\{
j\in\left[N_{i}-N_{i-1}-1\right]\mid v'_{N_{i-1}+j}>0\right\} .
\end{aligned}
\]
In practice, $\boldsymbol{v}$ will be one of the vectors of successive
differences $\boldsymbol{s}=\left(s_{1},\,s_{2},\,s_{3}\right)$
(cf.~\eqref{equ:sj}) and $\boldsymbol{v'}$ will be one of the vectors
of successive differences
$\boldsymbol{r}=\left(r_{1},\ldots,r_{8}\right)$
(cf.~\eqref{equ:rj}). Given a pair $(\mu,\nu)$ of partitions
satisfying $\nu\leq\overline{\mu}$, recall our definitions
\eqref{eq:0} and \eqref{eq:-1}.  It is easy to see that, for every
$i\in\left[r\right]$,
\begin{equation*}
\text{supp}_{i}^{M}(\boldsymbol{v}) 
=M_{i}-M_{i-1}-J_{i}^{\mu} \quad \textup{ and } \quad \text{supp}_{i}^{N}(\boldsymbol{v'}) 
=N_{i}-N_{i-1}-J_{i}^{\nu}.
\end{equation*}
Thus it follows that 
\begin{equation*}
\binom{M_{i}-M_{i-1}}{J_{i}^{\mu}}_{q^{-1}} 
=\binom{M_{i}-M_{i-1}}{\text{supp}_{i}^{M}(\boldsymbol{v})}_{q^{-1}}\textup{ and } \binom{N_{i}-N_{i-1}}{J_{i}^{\nu}}_{q^{-1}}
=\binom{N_{i}-N_{i-1}}{\text{supp}_{i}^{N}(\boldsymbol{v'})}_{q^{-1}}.
\label{eq:supp}
\end{equation*}
Let $\delta_{ij}$ be the usual Kronecker delta function. Substituting
Lemma \ref{lem:product B'} and \ref{lem:product A'}, rewriting the
expressions in terms of $r_{j}$ and $s_{j}$ and using \eqref{eq:supp},
we find that formula \eqref{def:Dv} for $D_v(q,t)$ splits into a
product as follows:
\begin{align*}
D_{v}(q,t)= &
\prod_{i=1}^{r}\binom{3-M_{i-1}}{3-M_{i}}_{q^{-1}}\binom{\widehat{M_{i}}-N_{i-1}}{\widehat{M_{i}}-N_{i}}_{q^{-1}}\prod_{i=1}^{r}A_{i}B_{i},
\end{align*}
where, for $i\in[r]$,
\begin{align*}
A_{i}=
&\sum_{r_{N_{i-1}+1}=0}^{\infty}\cdots\sum_{r_{N_{i}-1=0}}^{\infty}\sum_{r_{N_{i}}=1-\delta_{ir}}^{\infty}\binom{N_{i}-N_{i-1}}{\text{supp}_{i}^{N}(\boldsymbol{v'})}_{q^{-1}}\prod_{j=N_{i-1}+1}^{N_{i}}\left(q^{a_{j}}t^{\widecheck{M_{i}}+j}\right)^{r_{j}},\\ B_{i}=
&
\sum_{s_{M_{i-1}+1}=0}^{\infty}\cdots\sum_{s_{M_{i}-1=0}}^{\infty}\sum_{s_{M_{i}}=0}^{\infty}\binom{M_{i}-M_{i-1}}{\text{supp}_{i}^{M}(\boldsymbol{v})}_{q^{-1}}\prod_{j=M_{i-1}+1}^{M_{i}}\left(q^{b_{j}}t^{N_{i-1}+\widecheck{j}}\right)^{s_{j}}.
\end{align*}
	
We now show that all of the factors $A_{i}$ and $B_{i}$ are products
of Igusa functions and Gaussian binomial coefficients.  Given
$i\in[r]$ and $I\subseteq[M_{i}-M_{i-1}-1]$, we define ${\mathbf
  S}^{i}(I)$ to be the set of vectors
${\mathbf s}^{i}=\left(s_{M_{i-1}+1},\ldots,\,s_{M_{i}}\right)\in\mathbb{N}_{0}^{M_{i}-M_{i-1}}$
such that $s_{j}=0$ unless $j\in\left\{ M_{i-1}+k\mid k\in I\right\}
\cup\left\{ M_{i}\right\} $. With the numerical data defined in
\eqref{eq:num.data.Y}, we have
\[
\begin{aligned}
 B_{i} &
 \\=&\sum_{I\subseteq\left[M_{i}-M_{i-1}-1\right]}\binom{M_{i}-M_{i-1}}{I}_{q^{-1}}\sum_{{\mathbf
     s^{i}\in{\mathbf
       S}^{i}(I)}}\prod_{j\in\left(I+M_{i-1}\right)\cup\left\{
   M_{i}\right\}
 }\left(q^{b_{j}}t^{N_{i-1}+\widecheck{j}}\right)^{s_{j}}\\ =&\sum_{I\subseteq\left[M_{i}-M_{i-1}-1\right]}\binom{M_{i}-M_{i-1}}{I}_{q^{-1}}\left(\prod_{\iota\in
   I}\left(\sum_{s_{M_{i-1}+\iota}=1}^{\infty}\left(Y_{M_{i-1}+\iota}\right)^{s_{M_{i-1}+\iota}}\right)\right)\sum_{s_{M_{i}=0}}^{\infty}\left(Y_{M_{i}}\right)^{s_{M_{i}}}\\ =&\frac{1}{1-Y_{M_{i}}}\sum_{I\subseteq\left[M_{i}-M_{i-1}-1\right]}\binom{M_{i}-M_{i-1}}{I}_{q^{-1}}\prod_{\iota\in
   I}\frac{Y_{M_{i-1}+\iota}}{1-Y_{M_{i-1}+\iota}}\\ =&I_{M_{i}-M_{i-1}}\left(Y_{M_{i-1}+1},\,\ldots,\,Y_{M_{i}}\right).
\end{aligned}
\]
	
Analogously one shows that, with the numerical data defined in
\eqref{eq:num.data.X},
\[
A_{i}=\begin{cases}
I_{N_{i}-N_{i-1}}^{\circ}\left(X_{N_{i-1}+1},\dots,\,X_{N_{i}}\right)
& \text{for
}i<r,\\ I_{8-N_{r-1}}\left(X_{N_{r-1}+1},\dots,\,X_{8}\right) &
\text{for }i=r.
\end{cases}
\]
	
This completes the proof of
Theorem~\ref{thm:2D-reduction}. 
\end{proof}

\subsection{Explicit formulae for the functions $D_{v}$}\label{subsec:2dim}
\subsubsection{$v=\bfo^{3}\bft^{8}$}\label{subsubsec:2dim.1}

\[
D_{\bfo^{3}\bft^{8}}(q,t)=I_{3}\left(q^{2}t^{3},q^{2}t^{5},t^{6}\right)I_{8}((q^{i(8-i))}t^{6+i})_{i=1,\dots,8}).
\]

\begin{rem} 
 Note that
 ${\zeta_{\mathfrak{f}_{2,3}(\lri)}^{\idealgr}(s)}=\zeta_{\lri^3}(s)I_{3}\left(q^{2}t^{3},q^{2}t^{5},t^{6}\right)$;
 cf.\ Proposition~\ref{pro:2,3}.
\end{rem}

\subsubsection{$v=\bfo^{2}\bft^{j}\bfo\bft^{8-j}$, $j\in\{1,\dots,5\}$}\label{subsubsec:2dim.2}

\begin{multline*}
D_{\bfo^{2}\bft^{j}\bfo\bft^{8-j}}(q,t)=\binom{5}{j}_{q^{-1}}\binom{3}{1}_{q^{-1}}I_{2}\left(q^{2}t^{3},q^{2}t^{5}\right)\\ \cdot 
I_{j}^{\circ}((q^{2+i(5-i)}t^{5+i})_{i=1,\dots,j})I_{1}(q^{j(8-j)}t^{6+j})I_{8-j}((q^{i(8-i)}t^{6+i})_{i=j+1,\dots,8}).
\end{multline*}

\subsubsection{$v=\bfo\bft\bfo\bft^{j-1}\bfo\bft^{8-j}$, $j\in\{1,\dots,5\}$}\label{subsubsec:2dim.3}

For $j=1$, 
\[
D_{\bfo\bft\bfo^{2}\bft^{7}}(q,t)=\binom{3}{1}_{q^{-1}}\binom{2}{1}_{q^{-1}}I_{1}(q^{2}t^{3})
I_{1}^{\circ}(q^{3}t^{4})I_{2}(q^{6}t^{6},q^{7}t^{7})I_{7}((q^{i(8-i)}t^{6+i})_{i=2,\dots,8}).
\]
For $j\in\{2,\dots,5\}$, 
\begin{multline*}
D_{\bfo\bft\bfo\bft^{j-1}\bfo\bft^{8-j}}(q,t)=\binom{4}{j-1}_{q^{-1}}\binom{3}{1}_{q^{-1}}\binom{2}{1}_{q^{-1}}^{2}I_{1}(q^{2}t^{3})I_{1}^{\circ}(q^{3}t^{4})I_{1}(q^{6}t^{6})\\ \cdot
I_{j-1}^{\circ}((q^{2+i(5-i)}t^{5+i})_{i=2,\dots,j})I_{1}(q^{j(8-j)}t^{6+j})I_{8-j}((q^{i(8-i)}t^{6+i})_{i=j+1,\dots,8}).
\end{multline*}

\subsubsection{$v=\bfo\bft^{2}\bfo\bft^{j-1}\bfo\bft^{7-j}$, $j\in\{1,\dots,4\}$}\label{subsubsec:2dim.4}

For $j=1$,
\[
D_{\bfo\bft^{2}\bfo^{2}\bft^{6}}(q,t)=\binom{3}{1}_{q^{-1}}I_{1}(q^{2}t^{3})I_{2}^{\circ}(q^{3}t^{4},q^{2}t^{5})I_{2}(q^{8}t^{7},q^{12}t^{8})I_{6}((q^{i(8-i)}t^{6+i})_{i=3,\dots,8}).
\]
For $j\in\{2,3,4\}$, 
\begin{multline*}
D_{\bfo\bft^{2}\bfo\bft^{j-1}\bfo\bft^{7-j}}(q,t)=\binom{3}{j-1}_{q^{-1}}\binom{3}{1}_{q^{-1}}\binom{2}{1}_{q^{-1}}I_{1}(q^{2}t^{3})I_{2}^{\circ}(q^{3}t^{4},q^{2}t^{5})I_{1}(q^{8}t^{7})\\ \cdot
I_{j-1}^{\circ}((q^{2+i(5-i)}t^{5+i})_{i=3,\ldots,j+1})I_{1}(q^{(j+1)(8-(j+1))}t^{7+j})I_{8-(j+1)}((q^{i(8-i)}t^{6+i})_{i=j+2,\ldots,8}).
\end{multline*}

\subsection{Completion of the proof}
The first two claims of Theorem~\ref{thm:main} follows
from~\eqref{equ:2dim} and the explicit formulae for $D_v$, $v\in
\mcD^{(2)}$, given in Section~\ref{subsec:2dim}. To deduce the local
functional equation \eqref{equ:funeq}, one checks, repeatedly using
the functional equations \eqref{equ:funeq.igusa} for the Igusa
functions $I_h$ resp.\ $I^{\circ}_h$ and the well-known fact that, for
$a,b\in\N_0$, $a\geq b$, $\binom{a}{b}_q = q^{b(a-b)}
\binom{a}{b}_{q^{-1}}$, that each of the functions $D_{v}$ satisfies
\[D_{v}(q^{-1}t^{-1})=-q^{31}t^{20}D_{v}(q,t).\] 
The functional equation \eqref{equ:funeq} follows
from~\eqref{equ:2dim} as
$$\left.\zeta_{\lri^3}(s)\right|_{q\rarr q^{-1}} = -q^3 t^3 \zeta_{\lri^3}(s).$$

\section{Nilpotency class two ($c=2$)}

\label{sec:c=2}

Let $d\in\N_{\geq2}$. In this section we compute the local graded
ideal zeta functions
$\zeta^{\triangleleft_\textup{gr}}_{\mff_{d,2}(\lri)}(s)$. We prove
functional equations for these functions and establish their behaviour
at $s=0$. We use them to determine the abscissae of convergence of the
global graded ideal zeta functions
$\zeta^{\triangleleft_\textup{gr}}_{\mff_{d,2}(\Gri)}(s)$ and some
properties of the associated topological and reduced graded ideal zeta
functions. Throughout we write $d'$ for $\binom{d}{2}=W_{2}(d)$;
cf.\ \eqref{def:witt}.

\subsection{$\mfp$-Adic formulae}

The paper \cite{Voll/05a} determines the \emph{normal subgroup zeta
  functions} of the free class-$2$-nilpotent $d$-generator groups
$F_{2,d}$, enumerating these groups' normal subgroups of finite index.
By the Mal'cev correspondence, these are the {ideal zeta functions} of
the free nilpotent $\Z$-Lie rings $\mathfrak{f}_{2,d}$. The
computations generalize to the case of general number rings in a
straightforward manner.  To recall the paper's main result define the
function
\[
\phi:[d-1]_{0}\rarr[d'],\quad i\mapsto id-\binom{i+1}{2}.
\]
Given $(I,J)\in2^{[d-1]_{0}}\times2^{[d']}$, the paper
\cite[p.~71]{Voll/05a} defines a total order $\prec_{\phi(I),J}$ on
the disjoint union $I\cup J$.  Without loss of generality, we may
assume that the sets $I$ and $J$ have the same cardinality $h$. For
$i\in I=\left\{ i_{1},\ldots,i_{h}\right\} _{<}$ and $j\in J=\left\{
j_{1},\ldots,j_{h}\right\} _{<}$, set
\begin{align*}
j(i) & :=\min\left\{ j\in J\cup\left\{ d'\right\} \mid\phi(i)\prec_{\phi\left(I\right),J},j\right\} ,\\
i(j) & :=\max\left\{ i\in I\cup\left\{ 0\right\} \mid\phi(i)\prec_{\phi\left(I\right),J},j\right\} ,
\end{align*}
and
$j_{0}:=i_{0}:=0\prec_{\phi\left(I\right),J}\,\phi\left(\left[d-2\right]\right)\uplus\left[d'-1\right]\prec_{\phi\left(I\right),J}\,j_{h+1}:=d'$.
The following is essentially the main result of \cite{Voll/05a}, in
notation compatible with the current paper. (The underlining of terms
of the form $d+$ in \eqref{equ:num.data.ideal} is meant to facilitate
comparison with the ``graded'' numerical data
\eqref{equ:num.data.idealgr} in Theorem~\ref{thm:f2d.idealgr}, and may
be ignored.)

\begin{thm}\cite[Theorem 4]{Voll/05a}\label{thm:f2d.ideal}
For all primes $p$ and all finite extensions $\lri$ of $\Zp$,
\begin{equation}
\zeta_{\mff_{2,d}(\lri)}^{\vartriangleleft}(s)=\zeta_{\lri^{d}}(s)\sum_{\substack{I\subseteq[d-2],J\subseteq[d'-1]\\
\left|I\right|=\left|J\right|,\phi(I)\leq J
}
}\mcA_{I,J}^{\ideal}(q,t),\label{equ:f2d.ideal}
\end{equation}
where 
\begin{multline}
\mcA_{I,J}^{\ideal}(q,t)=\zeta_{\lri^{i_{h}}}(s)^{-1}I_{j_{1}}((X_{\alpha})_{\alpha\in[j_{1}-1]},X_{0})
\cdot \prod_{r=1}^{h}\binom{j_{r+1}-\phi(i_{r})}{j_{r}-\phi(i_{r})}_{q^{-1}}\binom{d-i_{r-1}}{d-i_{r}}_{q^{-1}}\\\cdot
I_{i_{r}-i_{r-1}}^{\circ}((Y_{\beta})_{\beta\in]i_{r-1},i_{r}[},Y'_{i_{r}})I_{j_{r+1}-j_{r}}((X_{\alpha})_{\alpha\in]j_{r},j_{r+1}[},X_{j_{r}}),\label{AIJ}
\end{multline}
with numerical data
\begin{alignat}{2}
X_{j} &
=q^{i\left(j\right)\left(d-i\left(j\right)\right)+\left(d'-j\right)\left(\ul{d+}j-\phi\left(i\left(j\right)\right)\right)}t^{d-i\left(j\right)+d'-j},
& j & \in\left[d'-1\right]_{0},\nonumber \\ Y_{i} &
=q^{i\left(d-i\right)+\left(d'-j\left(i\right)\right)\left(\ul{d+}j\left(i\right)-\phi\left(i\right)\right)}t^{d-i+d'-j\left(i\right)},
& i & \in\left[d-2\right],\label{equ:num.data.ideal}\\ Y'_{i_{r}} &
=q^{i_{r-1}\left(d-i_{r-1}\right)+\left(d'-j\left(i_{r}\right)\right)\left(\ul{d+}j\left(i_{r}\right)-\phi\left(i_{r-1}\right)\right)}t^{d-i_{r-1}+d'-j\left(i_{r}\right)},
& \quad r & \in\left[h\right].\nonumber
\end{alignat}
\end{thm}

The following is a graded analogue of
Theorem~\ref{thm:f2d.ideal}. 

\begin{thm}\label{thm:f2d.idealgr}
For all primes $p$ and all finite extensions $\lri$ of $\Zp$,
\begin{equation}
 \zeta_{\mff_{2,d}(\lri)}^{\idealgr}(s)=\zeta_{\lri^{d}}(s)\sum_{\substack{I\subseteq[d-2],J\subseteq[d'-1]\\ \left|I\right|=\left|J\right|,\phi(I)\leq
     J } }\mcA_{I,J}^{\idealgr}(q,t),\label{equ:f2d.idealgr}
\end{equation}
where $\mcA_{I,J}^{\idealgr}(q,t)$ is defined as $\mcA_{I,J}^{\ideal}(q,t)$
in~\eqref{AIJ}, but with numerical data 
\begin{alignat}{2}
X_{j} &
=q^{i\left(j\right)\left(d-i\left(j\right)\right)+\left(d'-j\right)\left(j-\phi\left(i\left(j\right)\right)\right)}t^{d-i\left(j\right)+d'-j},
& j & \in\left[d'-1\right]_{0},\nonumber \\ Y_{i} &
=q^{i\left(d-i\right)+\left(d'-j\left(i\right)\right)\left(j\left(i\right)-\phi\left(i\right)\right)}t^{d-i+d'-j\left(i\right)},
& i & \in\left[d-2\right],\label{equ:num.data.idealgr}\\ Y'_{i_{r}} &
=q^{i_{r-1}\left(d-i_{r-1}\right)+\left(d'-j\left(i_{r}\right)\right)\left(j\left(i_{r}\right)-\phi\left(i_{r-1}\right)\right)}t^{d-i_{r-1}+d'-j\left(i_{r}\right)},
&\quad r & \in\left[h\right].\nonumber
\end{alignat}
\end{thm}

\begin{proof} The difference between \eqref{equ:num.data.ideal}
and \eqref{equ:num.data.idealgr} reflects the difference between
\eqref{equ:gss.lemma.6.1} and \eqref{equ:gen.rewrite}. \end{proof}

\begin{cor} 
\begin{alignat}{2}
\left.\zeta_{\mff_{2,d}(\lri)}^{\ideal}(s)\right|_{q\rarr q^{-1}} & = & (-1)^{d+d'}q^{\binom{d+d'}{2}} & t^{2d+d'}\zeta_{\mff_{2,d}(\lri)}^{\ideal}(s),\label{equ:funeq.f2d.ideal}\\
\left.\zeta_{\mff_{2,d}(\lri)}^{\idealgr}(s)\right|_{q\rarr q^{-1}} & = & (-1)^{d+d'}q^{\binom{d}{2}+\binom{d'}{2}} & t^{2d+d'}\zeta_{\mff_{2,d}(\lri)}^{\idealgr}(s).\label{equ:funeq.f2d.idealgr}
\end{alignat}
\end{cor}

\begin{proof} Eq.~\eqref{equ:funeq.f2d.ideal} is
  \cite[Theorem~3(a)]{Voll/05a}.  The proof proceeds by establishing
  the symmetry in question for each of the summands in
  \eqref{equ:f2d.ideal}, using \eqref{equ:funeq.igusa}.  The formulae
  \eqref{equ:f2d.ideal} and \eqref{equ:f2d.idealgr} only differ in
  their numerical data. By Remark~\ref{rem:igusa}, it suffices to
  control the value of
  $X_{0}\cdot\prod_{r=1}^h\frac{X_{j_{r}}}{Y_{i_{r}}'}$ in both
  cases. Inspection of \eqref{equ:num.data.ideal} and
  \eqref{equ:num.data.idealgr}, respectively, reveals that in the
  ideal case this term is by a factor $q^{dd'}$ larger than in the
  graded ideal case. Thus
\[
\left.\zeta_{\mff_{2,d}(\lri)}^{\idealgr}(s)\right|_{q\rarr q^{-1}}=(-1)^{d+d'}q^{\binom{d+d'}{2}-dd'}t^{2d+d'}\zeta_{\mff_{2,d}(\lri)}^{\idealgr}(s).
\]
As $\binom{d+d'}{2}-dd'=\binom{d}{2}+\binom{d'}{2}$, this establishes~\eqref{equ:funeq.f2d.idealgr}.
\end{proof}

\subsection{$\mfp$-Adic behaviour at zero}
We describe the behaviours of $\zeta_{\mff_{2,d}(\lri)}^{\ideal}(s)$
and $\zeta_{\mff_{2,d}(\lri)}^{\idealgr}(s)$ at $s=0$.

\begin{thm}
For all primes $p$ and all finite extensions $\lri$ of $\Zp$, 
\begin{align}
\left.\frac{\zeta_{\mff_{2,d}(\lri)}^{\vartriangleleft}(s)}{\zeta_{\lri^{d+d'}}(s)}\right|_{s=0}
&
=1,\label{equ:spec.val.ideal}\\ \left.\frac{\zeta_{\mff_{2,d}(\lri)}^{\idealgr}(s)}{\zeta_{\lri^{d}}(s)\zeta_{\lri^{d'}}(s)}\right|_{s=0}
& =\frac{d\cdot
  d'}{d\cdot\left(d+d'\right)}=\frac{d-1}{d+1}.\label{equ:spec.val.idealgr}
\end{align}
\end{thm}

\begin{proof} 
 To prove \eqref{equ:spec.val.ideal} we note that, by
 \eqref{equ:f2d.ideal},
\[
 \frac{\zeta_{\mff_{2,d}(\lri)}^{\vartriangleleft}(s)}{\zeta_{\lri^{d+d'}}(s)}=\prod_{j=d}^{d+d'-1}(1-q^{j}t)\left(\mcA_{\varnothing,\varnothing}^{\ideal}(q,t)+\sum_{(I,J)\neq(\varnothing,\varnothing)}\mcA_{I,J}^{\ideal}(q,t)\right).
\]
By inspection of \eqref{equ:num.data.ideal} we see that
$\mcA_{\varnothing,\varnothing}^{\ideal}(q,t)=I_{d'}((q^{(d+(d'-j))j}t^{d+j})_{j\in[d']})$
and that
\[
\mcA_{I,J}^{\ideal}(q,1)=\begin{cases}
0 & \textrm{ if \ensuremath{(I,J)\neq(\varnothing,\varnothing)},}\\
I_{d'}((q^{(d+(d'-j))j})_{j\in[d']})=\prod_{j=d}^{d+d'-1}(1-q^{j}) & \textrm{ if }(I,J)=(\varnothing,\varnothing).
\end{cases}
\]

To prove \eqref{equ:spec.val.idealgr} we note that 
\[
\frac{\zeta_{\mff_{2,d}(\lri)}^{\idealgr}(s)}{\zeta_{\lri^{d}}(s)\zeta_{\lri^{d'}}(s)}=\prod_{j=0}^{d'-1}(1-q^{j}t)\left(\mcA_{\varnothing,\varnothing}^{\idealgr}(q,t)+\sum_{(I,J)\neq(\varnothing,\varnothing)}\mcA_{I,J}^{\idealgr}(q,t)\right).
\]
By inspection of \eqref{equ:num.data.idealgr} we see that
$\mcA_{\varnothing,\varnothing}^{\idealgr}(q,t)=I_{d'}((q^{(d'-j)j}t^{d+j})_{j\in[d']})$,
which has a simple pole at $t=1$. 
In contrast, $\mcA_{I,J}^{\idealgr}(q,t)$ has no pole at $t=1$ when
$(I,J)\neq(\varnothing,\varnothing)$. Hence 
\[
\left.\frac{\zeta_{\mff_{2,d}(\lri)}^{\idealgr}(s)}{\zeta_{\lri^{d}}(s)\zeta_{\lri^{d'}}(s)}\right|_{s=0}=\left.\frac{I_{d'}((q^{(d'-j)j}t^{d+j})_{j\in[d']})}{I_{d'}((q^{(d'-j)j}t^{j})_{j\in[d']})}\right|_{t=1}=\left.\frac{1-t^{d'}}{1-t^{d+d'}}\right|_{t=1}=\frac{d'}{d+d'}.\qedhere
\]
\end{proof}

\begin{rem} 
 In the pertinent special cases, \eqref{equ:spec.val.ideal} confirms
 \cite[Conjecture IV ($\mfP$-adic form)]{Rossmann/15}. 
\end{rem}

\subsection{Global analytic properties}
The following result compares some of the known analytic properties of
$\zeta^{\ideal}_{\mff_{2,d}(\Gri)}(s)$ (cf.\
\cite[Theorem~3]{Voll/05a}) with those of
$\zeta^{\triangleleft_\textup{gr}}_{\mff_{2,d}(\Gri)}(s)$.

\begin{thm} The abscissae of convergence of
  $\zeta^{\ideal}_{\mff_{2,d}(\Gri)}(s)$
  resp.\ $\zeta^{\idealgr}_{\mff_{2,d}(\Gri)}(s)$ are
\begin{align}
  \alpha^{\ideal}(\mff_{2,d}) &= \max \left\{ d,
    \frac{(\binom{d}{2}-j)(d+j)+1}{\binom{d+1}{2}-j} \mid
    j\in[\binom{d}{2}-1]\right\} \textup{ resp.\ }, \label{equ:abs.ungraded}\\
  \alpha^{\idealgr}(\mff_{2,d}) &= \max \left\{ d,
    \frac{(\binom{d}{2}-j)j+1}{\binom{d+1}{2}-j} \mid
    j\in[\binom{d}{2}-1]\right\}.\label{equ:abs.graded}
\end{align}
The respective meromorphic continuations of both zeta functions beyond
their abscissae of convergence have simple poles at
$s=\alpha^{\ideal}(\mff_{2,d})$
resp.\ $s=\alpha^{\idealgr}(\mff_{2,d})$.
\end{thm}

\begin{proof} Eq.\ \eqref{equ:abs.ungraded} and the ensuing claim about the
  meromorphic continuation is essentially
  \cite[Theorem~3]{Voll/05a}. Eq.\ \eqref{equ:abs.graded} and the
  analogous claim are proved analogously.
\end{proof}

\subsection{Topological zeta functions -- degree and behaviour at zero}

\begin{thm} 
\[
\deg_{s}\left(\zeta_{\mathfrak{f}_{2,d},\topo}^{\ideal}(s)\right)=-(d+d')=\deg_{s}\left(\zeta_{\mathfrak{f}_{2,d},\topo}^{\idealgr}(s)\right).
\]
Moreover,
\begin{align}
 \left.s\zeta_{\mff_{2,d},\topo}^{\vartriangleleft}(s)\right|_{s=0}
 &
 =\frac{\left(-1\right)^{d+d'-1}}{\left(d+d'-1\right)!},\label{equ:spec.val.ideal.top}\\ \left.s^{2}\zeta_{\mff_{2,d},\topo}^{\idealgr}(s)\right|_{s=0}
 &
 =\frac{\left(-1\right)^{\left(d-1\right)+\left(d'-1\right)}d'}{\left(d+d'\right)\left(d-1\right)!\left(d'-1\right)!}.\label{equ:spec.val.idealgr.top}
\end{align}
\end{thm}

\begin{proof} 
 Note that, given $h\in\N$ and, for $i\in[h]$,
 $X_{i}=q^{a_{i}}t^{b_{i}}$ for $a_{i}\in\N_{0}$, $b_{i}\in\N$,
\[
I_{h}(X_{1},\dots,X_{h})=\frac{h!}{\prod_{i=1}^{h}(b_{i}s-a_{i})}(q-1)^{-h}+O((q-1)^{-h+1}).
\]
The summands in the formula \eqref{equ:f2d.ideal} for
$\zeta_{\mff_{2,d}(\lri)}^{\ideal}(s)$ are all products of Igusa
functions and Gaussian binomial coefficients.  Hence there exist
$a_{I,J,i}^{\ideal}\in\N_{0}$ and
$b_{I,J,i}^{\ideal},c_{I,J}^{\ideal}\in\N$ such that
\[
\zeta_{\mff_{2,d}(\lri)}^{\ideal}(s)=\sum_{I,J}\frac{c_{I,J}^{\ideal}}{\prod_{i=1}^{d+d'}(b_{I,J,i}^{\ideal}s-a_{I,J,i}^{\ideal})}(q-1)^{-(d+d')}+O((q-1)^{-(d+d')+1}).
\]
(Here and in the sequel, the sums are over pairs $(I,J)$ as in~\eqref{equ:f2d.ideal}.)
Hence 
\begin{equation}
\zeta_{\mff_{2,d},\topo}^{\ideal}(s)=\sum_{I,J}\frac{c_{I,J}^{\ideal}}{\prod_{i=1}^{d+d'}(b_{I,J,i}^{\ideal}s-a_{I,J,i}^{\ideal})}\label{equ:f2d.ideal.top}
\end{equation}
is a rational function in $s$ of degree $-(d+d')$, confirming the
first claim.

The second claim, on the degree of
$\zeta_{\mff_{2,d},\topo}^{\idealgr}(s)$, follows from
analogous considerations based on \eqref{equ:f2d.idealgr}.  Indeed,
there exist $a_{I,J,i}^{\idealgr}\in\N_{0}$ and
$b_{I,J,i}^{\idealgr},c_{I,J}^{\idealgr}\in\N$ such that
\[
\zeta_{\mff_{2,d}(\lri)}^{\idealgr}(s)=\sum_{I,J}\frac{c_{I,J}^{\idealgr}}{\prod_{i=1}^{d+d'}(b_{I,J,i}^{\idealgr}s-a_{I,J,i}^{\idealgr})}(q-1)^{-(d+d')}+O((q-1)^{-(d+d')+1}).
\]
Hence 
\begin{equation}
\zeta_{\mff_{2,d},\topo}^{\idealgr}(s)=\sum_{I,J}\frac{c_{I,J}^{\idealgr}}{\prod_{i=1}^{d+d'}(b_{I,J,i}^{\idealgr}s-a_{I,J,i}^{\idealgr})}\label{equ:f2d.idealgr.top}
\end{equation}
is a rational function in $s$ of degree $-(d+d')$, too.

Turning to the behaviour at~$s=0$, we start with the observation that
$\zeta_{\lri^{d}}(s)\mcA_{I,J}^{\ideal}(q,t)$ has no pole at $t=1$
unless $(I,J)=(\vn,\vn)$, in which case it is simple. Thus the summand
$\frac{c_{I,J}^{\ideal}}{\prod_{i=1}^{d+d'}(b_{I,J,i}^{\ideal}s-a_{I,J,i}^{\ideal})}$
in \eqref{equ:f2d.ideal.top} has no pole at $s=0$ unless
$(I,J)=(\vn,\vn)$, in which case it is simple. Specifically,
\[
\zeta_{\lri^{d}}(s)\mcA_{\vn,\vn}^{\ideal}(q,t)=I_{d}((q^{i(d-i)}t^{i})_{i\in[d]})I_{d'}((q^{(d+d'-j)j}t^{d+j})_{j\in[d']}),
\]
whence 
\[
\frac{c_{\vn,\vn}^{\ideal}}{\prod_{i=1}^{d+d'}(b_{\vn,\vn,i}^{\ideal}s-a_{\vn,\vn,i}^{\ideal})}=\frac{1}{\prod_{i=0}^{d-1}(s-i)}\cdot\frac{(d')!}{\prod_{j=1}^{d'}((d+j)s-(d+d'-j)j)}
\]
and thus 
\[
\left.s \zeta_{\mff_{2,d},\topo}^{\ideal}(s)\right|_{s=0}=\frac{(d')!}{\prod_{i=1}^{d-1}(-i)\prod_{j=1}^{d'}(-(d+d'-j)j)}=\frac{(-1)^{d+d'-1}}{(d+d'-1)!},
\]
which establishes \eqref{equ:spec.val.ideal.top}.

The proof of \eqref{equ:spec.val.idealgr.top} goes along similar
lines. We observe that $\zeta_{\lri^{d}}(s)\mcA_{I,J}^{\idealgr}(q,t)$
has a simple pole at $t=1$ unless $(I,J)=(\vn,\vn)$, in which case
it is double. Thus the summand $\frac{c_{I,J}^{\idealgr}}{\prod_{i=1}^{d+d'}(b_{I,J,i}^{\idealgr}s-a_{I,J,i}^{\idealgr})}$
in \eqref{equ:f2d.idealgr.top} has a simple pole at $s=0$ unless
$(I,J)=(\vn,\vn)$, in which case it is double. Specifically, 
\[
 \zeta_{\lri^{d}}(s)\mcA_{\vn,\vn}^{\idealgr}(q,t)=I_{d}((q^{i(d-i)}t^{i})_{i\in[d]})I_{d'}((q^{(d'-j)j}t^{d+j})_{j\in[d']}),
\]
whence 
\[
 \frac{c_{\vn,\vn}^{\idealgr}}{\prod_{i=1}^{d+d'}(b_{\vn,\vn,i}^{\idealgr}s-a_{\vn,\vn,i}^{\idealgr})}=\frac{1}{\prod_{i=0}^{d-1}(s-i)}
 \cdot \frac{(d')!}{\prod_{j=1}^{d'}((d+j)s-(d'-j)j)}
\]
and thus
\[
\left.s^{2}\zeta_{\mff_{2,d},\topo}^{\idealgr}(s)\right|_{s=0}=\frac{(d')!}{(d+d')\prod_{i=1}^{d-1}(-i)\prod_{j=1}^{d'-1}(-(d'-j)j)}=\frac{(-1)^{(d-1)+(d'-1)}d'}{(d+d')(d-1)!(d'-1)!},
\]
which proves \eqref{equ:spec.val.idealgr.top}. \end{proof}

\begin{rem} 
 In the pertinent special cases, the theorem's first statement
 confirms \cite[Conjecture 1]{Rossmann/15}, whereas
 \eqref{equ:spec.val.ideal.top} confirms \cite[Conjecture IV
   (topological form)]{Rossmann/15}.
\end{rem}

\subsection{Explicit examples}

We record the following consequences of Theorem~\ref{thm:f2d.idealgr}.

\begin{pro}[$(c,d)=(2,2)$ -- Heisenberg]\label{pro:heisenberg.gr}
  \begin{align*}
    \zeta_{\mathfrak{f}_{2,2}(\lri)}^{\idealgr}(s) & =\zeta_{\lri^{2}}(s)I_{1}(t^{3})=:W^{\idealgr}_{2,2}(q,t),\\
    \zeta_{\mathfrak{f}_{2,2},\redu}^{\idealgr}(Y)=W^{\idealgr}_{2,2}(1,Y) & =\frac{1}{(1-Y)^{2}(1-Y^{3})},\\
    \zeta_{\mathfrak{f}_{2,2},\topo}^{\idealgr}(s) &
    =\frac{1}{3s^{2}(s-1)}.
\end{align*}
The global graded ideal zeta function
$\zeta^{\idealgr}_{\mff_{2,2}(\Gri)}(s)=\zeta_K(s)\zeta_K(s-1)\zeta_K(3s)$
has abscissa of convergence $2$ and meromorphic continuation to the
whole complex plane.
\end{pro}

\begin{pro}[$(c,d)=(2,3)$] \label{pro:2,3}
\begin{align*}
\zeta_{\mathfrak{f}_{2,3}(\lri)}^{\idealgr}(s) & =\zeta_{\lri^{3}}(s)I_{3}(q^{2}t^{3},q^{2}t^{5},t^{6})=:W^{\idealgr}_{2,3}(q,t),\\
\zeta_{\mathfrak{f}_{2,3},\redu}^{\idealgr}(Y)=W^{\idealgr}_{2,3}(1,Y) & =\frac{1+2Y^{3}+2Y^{5}+Y^{8}}{(1-Y)^{3}(1-Y^{3})(1-Y^{4})(1-Y^{5})(1-Y^{6})},\\
\zeta_{\mathfrak{f}_{2,3},\topo}^{\idealgr}(s) & =\frac{1}{s^{2}(s-1)(s-2)(3s-2)(5s-2)}.
\end{align*}
The global graded ideal zeta function $\zeta^{\idealgr}_{\mff_{2,3}(\Gri)}(s)$
has abscissa of convergence $3$ and may be continued meromorphically
to $\{s \in \C \mid \real(s) > 1/3\}$.
\end{pro}

\begin{rem}
 The formulae for $\zeta_{\mff_{2,2}(\lri)}^{\idealgr}$
 resp.\ $\zeta_{\mff_{2,3}(\lri)}^{\idealgr}$ (for almost all $p$) are
 also given in \cite[Table~2]{Rossmann/16} under the labels
 m3\textunderscore2 resp.\ m6\textunderscore2\textunderscore1.
\end{rem}

We omit the (largish) formula for
$\zeta_{\mff_{2,4}(\lri)}^{\idealgr}(s)$, but do note the following
consequences.

\begin{pro}[$(c,d)=(2,4)$] There exists a rational function
  $W_{2,4}(X,Y)\in\Q[X,Y]$ such that
\[
\zeta_{\mathfrak{f}_{2,4}(\lri)}^{\idealgr}(s)=W^{\idealgr}_{2,4}(q,q^{-s}).
\]
Setting 
\begin{multline*}
N_{2,4}(Y)=Y^{33}+5Y^{33}+12Y^{28}+6Y^{27}+15Y^{26}+26Y^{25}+11Y^{24}+39Y^{23}+40Y^{22}+43Y^{21}+\\ 62Y^{20}+45Y^{19}+66Y^{18}+61(Y^{17}+Y^{16})+66Y^{15}+45Y^{14}+62Y^{13}+\\ 43Y^{12}+40Y^{11}+39Y^{10}+11Y^{9}+26Y^{8}+15Y^{7}+6Y^{6}+12Y^{5}+5Y^{3}+1\in\Q[Y]
\end{multline*}
we have 
\begin{align*}
\zeta_{\mathfrak{f}_{2,4},\redu}^{\idealgr}(Y) & = W^{\idealgr}_{2,4}(1,Y)  = \frac{N_{2,4}(Y)}{(1-Y)^{4}(1-Y^{3})\prod_{i=6}^{10}(1-Y^{i})},\\
\zeta_{\mathfrak{f}_{2,4},\topo}^{\idealgr}(s) & =\frac{3}{10}\frac{36s^{2}-43s+12}{s^{2}(s-1)^{3}(s-2)(s-3)(3s-4)^{2}(2s-1)(7s-9)(9s-5)}.
\end{align*}
The global graded ideal zeta function $\zideal_{\mff_{2,4}(\Gri)}(s)$
has abscissa of convergence $4$ and may be continued meromorphically
to $\{s \in \C \mid \real(s) > 7/6\}$.
\end{pro}
\section{Two generators ($d=2$)}

\label{sec:d=2}

We compute the graded ideal zeta functions
$\zeta_{\mathfrak{f}_{c,2}(\lri)}^{\idealgr}(s)$ for $c\in\{3,4\}$ as
well as their reduced and topological counterparts. (For $c=1$ see
\eqref{equ:abelian}, for $c=2$ see
Proposition~\ref{pro:heisenberg.gr}.)

\begin{pro}{$((c,d)=(3,2))$} 
\label{pro:f32}
\begin{align*}
\zeta_{\mathfrak{f}_{3,2}(\lri)}^{\idealgr}(s) &
=\zeta_{\lri^2}(s) \frac{1+t^4}{(1-t^3)(1-qt^4)(1-t^5)}
=:W^{\idealgr}_{3,2}(q,t),\\ \zeta_{\mathfrak{f}_{3,2},\redu}^{\idealgr}(Y)
&= W^{\idealgr}_{3,2}(1,Y) = \frac{1+Y^{4}}{(1-Y)^{2}(1-Y^{3})(1-Y^{4})(1-Y^{5})},\\ \zeta_{\mathfrak{f}_{3,2},\topo}^{\idealgr}(s)
& =\frac{2}{15s^{3}(s-1)(4s-1)}.
\end{align*}
The global graded ideal zeta
function $$\zeta^{\triangleleft_\textup{gr}}_{\mff_{3,2}(\Gri)}(s) =
\frac{\zeta_K(s)\zeta_K(s-1)\zeta_K(3s)\zeta_K(4s)\zeta_K(4s-1)\zeta_K(5s)}{\zeta_K(8s)}$$
has abscissa of convergence $2$ and meromorphic continuation to the
whole complex plane.
\end{pro}

\begin{proof}(sketch)
  Let $\mff = \mff_{3,2}$ and recall the Hall basis $\mcH_{3,2}$ for
  $\mff$ from Table~\ref{tab:hall}.  Recall formula~\eqref{equ:X} for
  the zeta function of $\mff(\lri)$ in terms of pairs
  $(\Lambda_2,\Lambda_3)$ of sublattices of $\mff(\lri)^{(2)} \cong
  \lri$ and $\mff(\lri)^{(3)} \cong \lri^2$, respectively. Note that
  $|\mff(\lri)^{(2)}:\Lambda_2| = q^{\mu}$ for some $\mu\in\N_0$,
  $|\mff(\lri)^{(1)}:X(\Lambda_2)| = q^{2\mu}$, and
  $(\Lambda_2,\Lambda_3)$ satisfies the condition in \eqref{equ:X} if
  and only if
\begin{equation}\label{equ:cond.3,2}
\mfp^\mu \mff(\lri)^{(3)} \leq \Lambda_3.
\end{equation}
This proves the first claim; the others are trivial consequences.
\end{proof}

\begin{rem}\label{rem:f32}
  Comparison with
  $\zeta_{\mathfrak{f}_{3,2}(\lri)}^{\vartriangleleft}(s)$
  (cf.\ \cite[Theorem~2.35]{duSWoodward/08}) yields
  $\zeta_{\mathfrak{f}_{3,2},\redu}^{\idealgr}(Y)=\zeta_{\mathfrak{f}_{3,2},\text{red}}^{\vartriangleleft}(Y).$
  Note that the Hall basis $\mcH_{3,2}$ is nice and simple in the
  sense of \cite{Evseev/09}.  The formula for
  $\zeta_{\mff_{3,2}(\lri)}^{\idealgr}$ (for almost all $p$) is also
  given in \cite[Table~2]{Rossmann/16} under the label
  m5\textunderscore3\textunderscore1.\end{rem}

\begin{pro}{$((c,d)=(4,2))$} 
 Set
\begin{multline*}
 N_{4,2}(X,Y)=-X^{2}(Y^{22}+Y^{18}+Y^{17}+Y^{16}+Y^{15}+Y^{11})\\ -X(Y^{16}+Y^{15}-Y^{13}+Y^{9}-Y^{7}-Y^{6})+Y^{11}+Y^{7}+Y^{6}+Y^{5}+Y^{4}+1\in\Q[X,Y].
\end{multline*}
Then 
\begin{align*}
  \zeta_{\mathfrak{f}_{4,2}(\lri)}^{\idealgr}(s) &
  =\zeta_{\lri^2}(s)\frac{N_{4,2}(q,t)}{(1-t^{3})(1-qt^{4})(1-t^{5})(1-qt^{5})(1-q^{2}t^{6})(1-q^{2}t^{7})(1-t^{8})}\\ \zeta_{\mathfrak{f}_{4,2},\redu}^{\idealgr}(Y)
  &
  =\frac{Y^{17}+Y^{13}+2(Y^{12}+Y^{11}+Y^{10}+Y^{7}+Y^{6}+Y^{5})+Y^{4}+1}{(1-Y)^{2}\prod_{i=3}^{8}(1-Y^{i})},\\ \zeta_{\mathfrak{f}_{4,2},\topo}^{\idealgr}(s)
  & =\frac{1}{60}\frac{(20s-3)}{s^{4}(s-1)(5s-1)(4s-1)(7s-2)(3s-1)}.
\end{align*}
The global graded ideal zeta function $\zideal_{\mff_{4,2}(\Gri)}(s)$
has abscissa of convergence $2$ and may be continued meromorphically to $\{s \in \C
\mid \real(s) > 2/11\}$.
\end{pro}

\begin{proof}(sketch) 
 Let $\mff=\mff_{4,2}$ and recall the Hall basis $\mcH_{4,2}$ for
 $\mff$ from Table~\ref{tab:hall}.  Recall formula~\eqref{equ:X} for
 the zeta function of $\mff(\lri)$ in terms of triples
 $(\Lambda_2,\Lambda_3,\Lambda_4)$ of sublattices of $\mff(\lri)^{(2)}
 \cong \lri$, $\mff(\lri)^{(3)} \cong \lri^2$, and $\mff(\lri)^{(4)}
 \cong \lri^3$, respectively.  Note that $|\mff(\lri)^{(2)}:\Lambda_2|
 = q^{\mu}$ for some $\mu\in\N_0$, $|\mff(\lri)^{(1)}:X(\Lambda_2)| =
 q^{2\mu}$, and $(\Lambda_2,\Lambda_3,\Lambda_4)$ satisfies the
 condition in \eqref{equ:X} if and only if $(\Lambda_2,\Lambda_3)$
 satisfies \eqref{equ:cond.3,2} and $[\Lambda_3, \mff(\lri)^{(1)}] \leq
 \Lambda_4$.  To investigate the latter condition, we may assume that
 $\Lambda_3 = \la \pi^{\nu_1} xyx, \pi^{\nu_2}xyy \ra_{\lri}$ for a
 partition~$\nu = (\nu_1,\nu_2)_{\geq}$. Then
$$[\Lambda_3, \mff(\lri)^{(1)}] = \la \pi^{\nu_1}xyxx,
 \underline{\pi^{\nu_1}xyxy}, \pi^{\nu_2}xyyx, \pi^{\nu_2}xyyy
 \ra_{\lri}$$ As $xyxy = xyyx$ by the Jacobi identity, the underlined
 term may be omitted, whence
$$[\Lambda_3, \mff(\lri)^{(1)}] = \pi^{\nu_1}xyxx \oplus \pi^{\nu_2}xyyx
\oplus \pi^{\nu_2}xyyy.$$ To enumerate the lattices $\Lambda_4 \leq
\mff(\lri)^{(4)}$ satisfying $[\Lambda_3, \mff(\lri)^{(1)}] \leq
\Lambda_4$, we distinguish according to the ``overlap type'' such a
lattice may have with $[\Lambda_3, \mff(\lri)^{(1)}]$. Indeed, if
$\Lambda_4$ has elementary divisor type given by a partition
$(\xi_1,\xi_2,\xi_3)_{\geq}$, then either 
\begin{gather*}
\nu_1 \geq \nu_2 \geq \xi_1 \geq \xi_2 \geq
\xi_3  \quad\textup{ or }\quad \nu_1 \geq \xi_1
> \nu_2 \geq \xi_2 \geq \xi_3. 
\end{gather*}
Each of the two cases is covered by a formula similar to the one
established in Theorem~\ref{thm:2D-reduction}. Adding them yields a
formula for the zeta function, viz.\
\begin{multline*}
  \zeta^{\triangleleft_\textup{gr}}_{\mff_{4,2}(\lri)}(s) =
  \zeta_{\lri^2}(s)I_1(t^3) \\ \cdot\left( I_2(qt^4,t^5)
  I_3(q^2t^6,q^2t^7,t^8) +\binom{2}{1}_{q^{-1}}
  I_1(qt^4)I_1^\circ(qt^5) I_1(q^2t^6)I_2(q^2t^7,t^8)\right).
\end{multline*}
This proves the first claim; the others are trivial consequences. The
claim about the meromorphic continuation follows from
\cite[Lemma~5.5]{duSWoodward/08}.
\end{proof}

\section{General conjectures}

\label{sec:con} Let $c\in\N$ and $d\in\N_{\geq 2}$. We
record a number of general conjectures regarding the graded ideal zeta
functions $\zeta_{\mathfrak{f}_{c,d}(\lri)}^{\idealgr}(s)$ as well as
their topological and reduced counterparts.  We write $W$ for the Witt
function~$W_{d}$ (cf.\ \eqref{def:witt}) and set
$r=\rk_{\Z}(\mathfrak{f}_{c,d})=\sum_{i=1}^{c}W(i)$.  With the
exception of Conjecture~\ref{con:top.inf} for $c=2$ and $d \geq 5$,
all conjectures in this section are confirmed, in the relevant special
cases, by the results in this paper.

\begin{con}[Uniformity]\label{con:uniformity} 
  There exists $W_{c,d}^{\idealgr}(X,Y)\in\Q(X,Y)$ such that, for
  almost all primes $p$ and all finite extensions $\lri$ of $\Zp$,
\[
\zeta_{\mff_{c,d}(\lri)}^{\idealgr}(s)=W_{c,d}^{\idealgr}(q,q^{-s}).
\]
\end{con}

\begin{con}[Local functional equations]\label{con:funeq} 
  For almost all primes $p$ and all finite extensions $\lri$ of $\Zp$,
\[
\zeta_{\mff_{c,d}(\lri)}^{\idealgr}(s)|_{q\rarr q^{-1}}=(-1)^{r}q^{\sum_{i=1}^{c}\binom{W(i)}{2}-\sum_{i=1}^{c}(c+1-i)W(i)s}\zeta_{\mff_{c,d}(\lri)}^{\idealgr}(s).
\]
\end{con}

\begin{rem}\label{rem:funeq}
 Local functional equations are a universal feature for quite general
 subring zeta functions (\cite[Theorem~A]{Voll/10}) and for ideal zeta
 functions of nilpotent Lie rings of nilpotency class $2$
 (\cite[Theorem~C]{Voll/10}). In fact, we expect that the proof of the
 latter result should carry over to the graded setting, making the
 functional equations \eqref{equ:funeq.f2d.idealgr} an instance of a
 general phenomenon in class~$2$; cf.\ Example~\ref{exm:ideal.z.f}.
 If Conjecture~\ref{con:uniformity} holds, then
 Conjecture~\ref{con:funeq} states that
\[
W_{c,d}^{\idealgr}(X^{-1},Y^{-1})=(-1)^{r}X^{\sum_{i=1}^{c}\binom{W(i)}{2}}Y^{\sum_{i=1}^{c}(c+1-i)W(i)}W_{c,d}^{\idealgr}(X,Y).
\]
 In general, the operation $q\rarr q^{-1}$ is defined rather more
 delicately in terms of inversion of Frobenius eigenvalues.

 In nilpotency class greater than $2$ local functional equations are
 not to be expected in general for ideal zeta functions of nilpotent
 Lie rings: see \cite[Section~1.2.3]{duSWoodward/08} for
 counterexamples in the ungraded setting and
 \cite[Table~2]{Rossmann/16} in the graded setting.

  A sufficient criterion for local functional equations in an
  enumerative setup generalizing ideal zeta functions of nilpotent Lie
  rings is given in \cite[Theorem~1.2]{Voll/16}. It applies to the
  ideal zeta functions of the free nilpotent Lie rings $\mff_{c,d}$;
  \cite[Theorem~4.4]{Voll/16}. Note that, in the notation of this
  result, $N_0=r$ and $\sum_{i=0}^{c-1} N_i = \sum_{i=1}^c(c+1-i)
  W(i)$.
\end{rem}

For the following conjecture, assume that Conjecture~\ref{con:uniformity}
holds. In this case, we may define the \emph{reduced graded ideal
zeta function} 
\[
\zeta_{\mff_{c,d},\redu}^{\idealgr}(Y)=W_{c,d}^{\idealgr}(1,Y)\in\Q(Y).
\]
We expect that the technology in \cite{Evseev/09} may be adapted to
define the reduced graded ideal zeta function even without the
assumption of Conjecture~\ref{con:uniformity}.

\begin{con}[Reduced zeta function]\label{con:red} The reduced
graded ideal zeta function $\zeta_{\mff_{c,d},\redu}^{\idealgr}(Y)$
has a pole of order $r$ at $1$ and there exist $e_{i}\in\N$,
$i\in[r]$, as well as \emph{nonnegative} integers $a_{j}\in\N_{0}$, $j
\in [N]_0$, where $N:=\left(\sum_{i=1}^{r}e_{i}\right) -
\sum_{j=1}^c(c+1-j) W_d(j)$, with $a_0=1$ and
\begin{equation}
a_{j}=a_{N-j}\textrm{ for all }j\in[N]_{0}\textrm{ and }\label{equ:palindromy}
\end{equation}
\begin{equation*}
\zeta_{\mff_{c,d},\redu}^{\idealgr}(Y)=\frac{\sum_{j=0}^{N}a_{j}Y^{j}}{\prod_{i=1}^{r}(1-Y^{e_{i}})}.\label{equ:con.red}
\end{equation*}
\end{con}

\begin{rem}\label{rem:red} The palindromic property~\eqref{equ:palindromy}
holds, of course, in particular if Conjecture~\ref{con:funeq} holds.
What Conjecture~\ref{con:red} suggests is that the reduced graded
ideal zeta functions $\zeta_{\mff_{c,d},\redu}^{\idealgr}(Y)$ share
key properties with the Hilbert-Poincar\'e series of graded
Cohen-Macaulay algebras. If there were such an algebra for each $c,d$,
its Cohen-Macaulayness would explain the nonnegativity of the
coefficients $a_{j}$, whereas their palindromy would reflect its
Gorensteinness (assuming that it is a domain; cf.\ \cite{Stanley/78}).

For $c=2$, the validity of Conjecture~\ref{con:red} may be seen from
the facts that
$\zeta_{\mathfrak{f}_{2,d,\redu}}^{\idealgr}(Y)=\zeta_{\mathfrak{f}_{2,d,\redu}}^{\ideal}(Y)$
and the Hall basis $\mcH_{2,d}$ for $\mff_{2,d}$ in
Table~\ref{tab:hall} is nice and simple in the sense
of~\cite{Evseev/09}.  Hence \cite[Proposition~4.1]{Evseev/09} is
applicable. \end{rem}

The following is a graded analogue of \cite[Conjecture~I]{Rossmann/15}.

\begin{con}[Degree of topological zeta function] 
\[
\deg_{s}\left(\zeta_{\mff_{c,d},\topo}^{\idealgr}(s)\right)=-r.
\]
\end{con}

\begin{con}[Behaviour of topological zeta function at $\infty$]\label{con:top.inf}
\begin{align*}
\left.s^{-r}\zeta_{\mff_{c,d},\topo}^{\idealgr}(s^{-1})\right|_{s=0}
&
=\left.(1-Y)^{r}\zeta_{\mff_{c,d},\redu}^{\idealgr}(Y)\right|_{Y=1}\in\Q_{>0}.
\end{align*}
\end{con}

\begin{rem} 
 We have no interpretation in terms of $c$ and $d$ for the numbers
 featuring in Conjecture~\ref{con:top.inf}. If
 Conjecture~\ref{con:red} were to hold for the reasons we explained in
 Remark~\ref{rem:red}, they were related to the ``multiplicity'' of
 the associated graded algebra.
\end{rem}

We close with two conjectures on the behaviour of $\mathfrak{p}$-adic
and topological graded ideal zeta functions at $s=0$.

\begin{con}[Behaviour of topological zeta function at zero]\label{con:top.zero}
  The topological gra\-ded ideal zeta function
  $\zeta_{\mff_{c,d},\topo}^{\idealgr}(s)$ has a pole
  of order $c$ at $s=0$ with leading coefficient
\begin{equation*}
  \left.s^{c}\zeta_{\mff_{c,d},\topo}^{\idealgr}(s)\right|_{s=0}=\frac{(-1)^{\sum_{i=1}^{c}{(W(i)-1)}}\prod_{i=1}^{c}W(i)}{\prod_{i=1}^{c}\left(\sum_{j=1}^{i}W(j)\right)\left(W(i)-1\right)!}.
\end{equation*}
\end{con}

\begin{rem} \cite[Conjecture IV (topological form)]{Rossmann/15}
asserts that
$\zeta_{\mff_{c,d},\topo}^{\vartriangleleft}(s)$ has a
simple pole at $s=0$ with residue
\begin{equation*}
\left.s\zeta_{\mff_{c,d},\topo}^{\vartriangleleft}(s)\right|_{s=0}=\frac{\left(-1\right)^{r-1}}{\left(r-1\right)!}.
\end{equation*}
\end{rem}

\begin{con}[Behaviour of $\mfp$-adic zeta function at zero]\label{con:p-ad.zero}
\begin{equation}\label{equ:con.padic.rhs}
 \left.\frac{\zeta_{\mathfrak{f}_{c,d}(\lri)}^{\idealgr}(s)}{\prod_{i=1}^{c}\zeta_{\lri^{W(i)}}(s)}\right|_{s=0}=\frac{\prod_{i=1}^{c}W(i)}{\prod_{i=1}^{c}\left(\sum_{j=1}^{i}W(j)\right)}.
\end{equation}
\end{con}

\begin{rem} \cite[Conjecture IV ($\mathfrak{P}$-adic form)]{Rossmann/15}
asserts that
\begin{equation}\label{equ:con.padic.tr.rhs}
\left.\frac{\zeta_{\mff_{c,d}(\lri)}^{\ideal}(s)}{\zeta_{\lri^{r}}(s)}\right|_{s=0}=1.
\end{equation}
It seems remarkable that the right hand sides of both
\eqref{equ:con.padic.rhs} and \eqref{equ:con.padic.tr.rhs} are
independent of~$\lri$. Note that $s=0$ is outside the domain of
convergence of the series defining
$\zeta_{\mff_{c,d}(\lri)}^{\ideal}(s)$
resp.~$\zeta_{\mff_{c,d}(\lri)}^{\idealgr}(s)$.
\end{rem}

\begin{rem}
  Simply replacing the numbers $W(i)$ by the ranks of the successive
  quotients of the upper central series does not extend
  Conjectures~\ref{con:top.zero} and \ref{con:p-ad.zero} to graded
  ideal zeta functions of general (non-free) nilpotent Lie
  rings. Consider, for instance, the direct product $L :=
  \mff_{2,2}\times \mff_{2,2}$. Then, for all primes $p$ and all
  finite extensions $\lri$ of $\Zp$,
\begin{equation}\label{equ:heis.square}
\zeta^{\idealgr}_{L(\lri)}(s) =
\zeta_{\lri^4}(s)\frac{1-t^5}{(1-t^3)^2(1-qt^5)}.
\end{equation}
Indeed, $L$ is the ``fundamental graded Lie ring''
m6\textunderscore2\textunderscore3, whose local graded ideal zeta
functions are recorded (for almost all $p$) in~\cite[Section~10,
  Table~2]{Rossmann/16}. Formula \eqref{equ:heis.square} may also
easily be derived directly from formulae for the ideal zeta functions
of $L$ (cf.\ \cite[Theorem~2.4]{duSWoodward/08}), using the comparison
identities recorded in Example~\ref{exm:ideal.z.f}.  In any case,
$$\left.s^2 \zeta^{\idealgr}_{L,\topo}(s)\right|_{s=0} = \frac{5}{54} \neq
\frac{(-1)^{3+1}4\cdot 2}{4\cdot 6 \cdot 3! \cdot 1 !} =
\frac{1}{18}.$$
Likewise,
$$\left.\frac{\zeta_{L(\lri)}^{\idealgr}(s)}{\zeta_{\lri^4}(s)\zeta_{\lri^2}(s)}\right|_{s=0}=\frac{5}{9}\neq
\frac{4\cdot 2}{4\cdot 6}=\frac{1}{3}.$$ That this value, however, is
a nonnegative rational number independent of $\lri$, as well as the
coincidence
$$\left.s^{-6}\zeta_{L, \topo}^{\idealgr}(s^{-1})\right|_{s=0} =
\frac{1}{9} =
\left.(1-Y)^{6}\zeta_{L,{\redu}}^{\idealgr}(Y)\right|_{Y=1}\in\Q_{>0},$$
seem to indicate that some of the conjectures made in this section may
have generalizations to more general graded ideal zeta functions.
\end{rem}

\section*{Acknowledgments}
Our work was supported by DFG Sonderforschungsbereich 701 ``Spectral
Structures and Topological Methods in Mathematics'' at Bielefeld
University. The first author was also supported by A23200000 fund from
the National Institute for Mathematical Sciences. We acknowledge
numerous helpful conversations with Tobias Rossmann.


\def\cprime{$'$}
\providecommand{\bysame}{\leavevmode\hbox to3em{\hrulefill}\thinspace}
\providecommand{\MR}{\relax\ifhmode\unskip\space\fi MR }
\providecommand{\MRhref}[2]{%
  \href{http://www.ams.org/mathscinet-getitem?mr=#1}{#2}
}
\providecommand{\href}[2]{#2}

\end{document}